\documentclass[a4paper]{amsart}
\usepackage{amsmath,amsthm,amssymb,latexsym,epic,bbm,comment,mathrsfs}
\usepackage{graphicx,enumerate,stmaryrd, xcolor, color}
\usepackage[all,2cell]{xy}
\xyoption{2cell}

\theoremstyle{plain}
\newtheorem{thm}{Theorem}%[section]
\newtheorem*{thm*}{Theorem}
\newtheorem*{thmA}{Theorem A}
\newtheorem*{thmB}{Theorem B}
\newtheorem*{thmC}{Theorem C}

\newtheorem{lem}[thm]{Lemma}
\newtheorem{prop}[thm]{Proposition}

\newtheorem{cor}[thm]{Corollary}
\newtheorem{df-prop}[thm]{Definition-Proposition}

\theoremstyle{definition}

\theoremstyle{remark}
\newtheorem{rem}[thm]{Remark}
\newtheorem{ex}[thm]{Example}

\usepackage[all]{xy}
\usepackage[active]{srcltx}
\usepackage[parfill]{parskip}
\usepackage{enumerate}

\newcommand{\Hom}{\operatorname{Hom}}

\usepackage{hyperref}
\newcommand{\mc}{\mathcal}
\newcommand{\mf}{\mathfrak}
\newcommand{\C}{\mathbb C}

\newcommand{\oa}{{\bar 0}}
\newcommand{\ob}{{\bar 1}}
\newcommand{\vare}{\epsilon} %%%% change-original \vere=\varepsilon
\def\gl{\mathfrak{gl}}
\newcommand{\g}{\mathfrak{g}}
\def\Ann{{\text{Ann}}}
\def\la{\lambda}
\def\pn{\mf{pe} (n)}
\def\ov{\overline}
\newcommand{\ch}{\mathrm{ch}}
\newcommand{\h}{\mathfrak{h}}

\newcommand{\Z}{{\mathbb Z}}
\newcommand{\rad}{\mathrm{rad}}
\def\mod{\operatorname{-mod}\nolimits}

\def\Hom{\operatorname{Hom}\nolimits}

\def\Res{\operatorname{Res}\nolimits}
\def\Ind{\operatorname{Ind}\nolimits}

\def\pr{\operatorname{pr}\nolimits}

\def\gl{\mathfrak{gl}}

\def\la{\lambda}

\def\pn{\mf{pe} (n)}
\def\ov{\overline}
\newcommand{\ad}{\mathrm{ad}}
\begin{document}
\title[Kostant's problem for Whittaker modules]
{Kostant's problem for Whittaker modules}

\author[Chen]{Chih-Whi Chen}
\address{Department of Mathematics, National Central University, Chung-Li, Taiwan 32054} \email{cwchen@math.ncu.edu.tw}

\begin{abstract}
	We study the classical problem of Kostant for Whittaker modules over Lie algebras and Lie superalgebras. 	We give a sufficient condition for a positive answer
	to Kostant's problem for the standard Whittaker modules over reductive Lie algebras. Under the same condition,  the positivity of the answer for simple Whittaker modules is reduced to that for simple highest weight modules. We develop several reduction  results to reduce the  Kostant's problem for standard and simple Whittaker modules over a type I Lie superalgebra  to that for the corresponding Whittaker modules over the even part of this Lie superalgebra.
\end{abstract}

\maketitle



\vspace{5mm}

\section{Introduction and background}\label{sec1}
\subsection{Kostant's problem}   Let $\g$ be a complex reductive finite-dimensional Lie algebra and $U(\g)$ its universal enveloping algebra.  For any $\g$-module $M$ the space $\Hom_\C(M,M)$ can be viewed as a $U(\g)$-bimodule in the natural
way. Following \cite[Kapitel 6]{Ja}, we denote  by $\mc L(M,M)$ the subbimodule of  $\Hom_\C(M,M)$ consisting of all elements   the adjoint action of $\g$ on which is locally finite and semisimple. Since   $U(\g)$ is locally finite under the adjoint action of $\g$, the representation map  $\phi: 
U(\g) \rightarrow  \mc L(M,M)$ is a homomorphism, and the kernel of this map is the annihilator $\Ann_{U(\g)}(M)$
of $M$ in $U(\g)$. %with a fixed triangular decomposition $\g =\mf n_- \oplus \h \oplus \mf n_+$. 

The classical problem of Kostant, as  popularized  by Joseph in \cite{Jo} in the setup of semisimple Lie algebras, is a famous open
problem in representation theory and is  formulated in the following way: \vskip0.2cm

{\bf Kostant's problem}. For which $\g$-module $M$ is the representation map $\phi$ 
\begin{align*}
	&\phi: U(\g) \rightarrow \mc L(M,M),
\end{align*} surjective? \vskip0.2cm

The positive answer to Kostant’s problem provides   important tools, in particular, in the study of equivalences of various categories for Lie algebra modules and the study of induced modules; see, e.g., \cite{MS, KhM04, MaSt08}.
In recent years, Kostant's problem  has been studied extensively for weight modules in the BGG category $\mc O$ from \cite{BGG}. The answer is known to be positive for Verma modules (see  \cite[Corollary 6.4]{Jo} and \cite[Corollary 7.25]{Ja}) and for certain classes of simple highest weight modules,  whose highest
weights are associated to parabolic subgroups of the Weyl group (see \cite{Jo,GaJo81, Ja}). However, it is also known that there exist simple highest weight modules for which the answer is negative; see, e.g., \cite[Section 9.5]{Jo}. 

  Mazorchuk \cite{Ma05} in 2005 developed    a new elegant approach to this problem that uses Arkhipov's twisting functor from \cite{Ar04} (see also  \cite{AS}). Since then, there have been  various refinements of this approach %using twisting functors 
  employed in \cite{MaSt08,Ka10, KaM10}  and then gave more positive and negative answers to  Kostant's problem for simple highest weight modules. However, a conjectural complete list of the answers to Kostant's problem was not even  available for modules in the principal block of the BGG  category $\mc O$ until the work of  Ko, Mazorchuk and Mrden \cite{KMM23}, which was partially   motivated by K\aa hrstr{\" o}m's conjectures raised in his email to Mazorchuk; see \cite[Section 1.2]{KMM23}.  %In loc. cit. it is proved that  Kostant’s problem is equivalent to  some homological problems of decomposing translated simple modules and then established a conjectural answer to Kostant's problem as well as the validity of one of K\aa hrstr{\" o}m's conjectures. % {\color{red} However, a conjectural answer to  Kostant's problem was not even  available for modules in the principal block $\mc O_0$ of the BGG  category $\mc O$ until quite recently. K\aa hrstr{\" o}m in 2019 formulated several conjectures about the complete answer for simple module in $\mc O_0$ in terms of the structure of the translated simple modules; see \cite[Conjecture 1.2]{KMM23}. Subsequently, Ko, Mazorchuk and Mrden \cite{KMM23} provided a conjectural answer to Kostant's problem in terms of Kazhdan-Lusztig basis and proved one of K\aa hrstr{\" o}m's conjectures.} 
    We also refer to \cite{Co74,Jo2, MaSt08a, Ma10, MaMe11, MMM22,Ma23, MSr23} and references therein for more results on Kostant's problem for simple highest weight modules and parabolic Verma modules.

While substantial progress for Kostant's problem for weight modules  has been made  by now towards a satisfactory theory and a complete answer, the problem for non-weight modules has not been studied much. As was already observed in \cite{Jo2}, Kostant's problem might have a negative answer for Kostant's simple Whittaker $\mf{sl}(2)$-modules from \cite{Ko78}, in contrast to the positivity of the answer for simple highest weight $\mf{sl}(2)$-modules. 

\subsection{Whittaker modules}
	\subsubsection{Whittaker modules over Lie algebras} Let $\g$ be a finite-dimensional complex reductive  Lie algebra with a fixed triangular decomposition  \begin{align}
		&\g=\mf n_-\oplus \mf h\oplus\mf n_+ \label{defn::trian}
	\end{align} with Cartan subalgebra $\mf h$ and nilradicals $\mf n_{\pm}$. Let $Z(\g)$ be the center of $U(\g)$. In the present paper, we refer to the finitely generated $\g$-modules that are locally finite over the $Z(\mf g)$ and $U(\mf n_+)$ as {\em Whittaker modules}.  Well-known examples include modules in the BGG category $\mc O$  with respect to the Borel subalgebra $\mf h+\mf n_+$.

 	Motivated by the theory of Whittaker models, in his 1978 seminal paper \cite{Ko78} Kostant introduced  a family of simple modules $Y_{\xi,\eta}$ over $\g$. %, including earlier $\mf{sl}(2)$-example .  
 	 In loc. cit., Kostant characterized these modules as simple Whittaker modules containing the so-called  Whittaker vectors associated to  {\em non-singular} characters $\zeta: \mf n_+\rightarrow \C$, that is, $\zeta$ do not vanish on any simple root vector. Subsequently, a systematic construction of Whittaker modules has been further developed by McDowell   \cite{Mc,Mc2} and
 	 by Mili{\v{c}}i{\'c} and Soergel \cite{MS, MS2} in full generality.    We also refer to  \cite{Mat88, B, BM, We11, CoM, Bro,  BR20, AB, R1,BR} and references therein for more   results on the study of Whittaker modules over reductive Lie algebras.
 	 
 	  The most basic  amongst the Whittaker modules is the so-called {\em standard Whittaker modules}, which are denoted by $M(\la,\zeta)$ with parameters $(\la,\zeta)\in \h^\ast\times \ch \mf n_+$. In particular, in the case that $\zeta=0$ they  coincide with Verma modules. In general, the isomorphism classes of  simple Whittaker modules can be represented  as the simple tops of $M(\la,\zeta)$ with $(\la,\zeta)\in \h^\ast\times \ch \mf n_+$, which we denote by $L(\la,\zeta)$. 
 	 
 	   The standard  Whittaker modules plays an important role in various investigations, and this has led to  numerous remarkable applications. In \cite{B}, Backelin developed a complete solution to the problem of composition multiplicity of the standard Whittaker modules in terms of Kazhdan–Lusztig polynomials, extending earlier work of Mili{\v{c}}i{\'c} and Soergel \cite{MS}. Losev in \cite[Theorems 4.1, 4.3]{Lo09} established a  remarkable equivalence between the category of Whittaker modules and a version of category $\mc O$ for a certain finite $W$-algebra and this, together with Backelin's solution, prove an earlier conjecture of Brundan and Kleshchev  \cite[Conjecture 7.17]{BK08} about the multiplicity problem of Verma modules over certain finite $W$-algebras.

The positive answer to Kostant's problem for  (thick) standard Whittaker modules over Lie algebras plays a significant role in the approach of establishing equivalences between various categories of Whittaker modules and Harish-Chandra bimodules proposed in \cite{MS}; see also \cite{Ch21, Ch212, CC22,CCM,CC23} for its direct application to representation theory of Lie superalgebras and finite $W$-superalgebras.

	\subsubsection{Whittaker modules over   Lie superalgebras}
	 A finite-dimensional Lie superalgebra $\widetilde{\mf g}=\widetilde{\g}_\oa\oplus \widetilde{\g}_\ob$ is called {\em quasi-reductive} if $\widetilde{\g}_\oa$
	is reductive and $\widetilde{\g}_\ob$	is semisimple as a $\widetilde \g_\oa$-module under the adjoint action; see also \cite{Se11}. 
	Such Lie superalgebras are called classical in \cite{Ma, CCC}.    Similar to the Lie algebra module case, knowing the answer to Kostant's problem for a given module over Lie superalgebra is important for using the approach of constructing equivalences of categories for
		Lie superalgebra modules; see \cite{Go, MaMe12, ChCo}.
	
	Quasi-reductive Lie superalgebras admit the notion of triangular decompositions; see also \cite{Mu12, Ma, CCC}. Recently, the theory of Whittaker modules over quasi-reductive Lie superalgebras and its applications to other fields  have been systematically developed by the various authors in \cite{BCW, Ch21, Ch212, CC22, CCM, CC23}. In particular, the standard Whittaker modules can be  defined, in a natural way, in the general setup of quasi-reductive Lie superalgebras and play a significant role in the study of Whittaker modules.

		While there has been considerable recent progress in the Kostant's problem for modules over Lie algebras, at present, analogous problem has received less attention in the setup of Lie superalgebras. For {\em basic classical} Lie superalgebras (in the sense of Kac \cite{Ka1}), Gorelik in \cite{Go} established positive answers to Kostant's problem for all Verma modules of {\em strongly typical} highest weights. An analogous result  for the {\em periplectic Lie superalgebra}  $\pn$ (in the sense of Kac \cite{Ka1}) has also been proved by Serganova in \cite{Se02}. To the
		best of our knowledge, a complete solution to this problem has not even been obtained for Verma modules over an arbitrary basic classical or strange Lie superalgebra, in contrast to the Lie algebra case.

 \subsection{Goal} The major motivation for the present paper is our attempt to determine answers to Kostant's problem for  standard and simple Whittaker modules over Lie algebras and Lie superalgebras. Special attention
	is paid to the type I Lie superalgebras, including reductive Lie algebras, the general linear Lie superalgebras $\gl(m|n)$, the ortho-symplectic Lie superalgebras $ \mf{osp}(2|2n)$ and the periplectic Lie superalgebras $\pn$ from Kac's list  \cite{Ka1} throughout.

\subsection{The main results} 
Throughout the paper the symbols $\Z$, $\Z_{>0}$ and $\Z_{<0}$ stand for the sets of all, positive and negative integers, respectively. All vector spaces, algebras, tensor products, et cetera, are over the field of complex numbers $\C$.
\subsubsection{} \label{sect::themainresults}  To explain the main results of the paper in more detail, we start by explaining our precise setup. 
%{\color{cyan}Let $\g$ be a semi-simple complex Lie algebra with a fixed triangular decomposition  \begin{align} &\g=\mf n_-\oplus \mf h\oplus\mf n_+. \label{defn::trian} \end{align} and $U(\g)$ be the universal enveloping algebra of $\g$.}  Let $\Phi$ be the set of all roots  with respect to the triangular decomposition in  \eqref{defn::trian}. 
Recall that we fix a triangular decomposition  in  \eqref{defn::trian} with Cartan subalgebra $\mf h$ and nilradicals $\mf n_{\pm}$. Let $\Phi\subset \mf h^\ast$ denote the corresponding set of roots.  Denote by $\Pi, \Phi^+$ and $\Phi^-$  
the sets of simple, positive and negative roots,  respectively. For a given root $\alpha\in \Phi$, we denote by $\g^\alpha:=\{x\in \g|~hx=\alpha(h)x, \text{for all }h\in \h\}$ the
corresponding root space. Denote by $\ch \mf n_+  = (\mf n_+/[\mf n_+,\mf n_+])^\ast$ the set of all  characters of $\mf n_+.$  For a given  character $\zeta\in \ch \mf n_+$, we  define a subset $\Pi_\zeta\subseteq \Pi$ as follows: 
\[\Pi_\zeta:=\{\alpha\in \Pi|~\zeta(\g^\alpha)\neq 0\}.\]
 This gives rise to a Levi subalgebra $\mf l_\zeta$ generated by $\g^\alpha$ ($\pm \alpha \in \Pi_\zeta$) and $\h$.

The Weyl group $W$ of $\g$ acts on $\mf h^\ast$ naturally. It is more convenient to consider the dot-action given by $w\cdot \la =w(\la+\rho)-\rho$, for $w\in W$ and $\la \in \h^\ast$, where $\rho$ is the half sum of the positive roots. Denote by $W_\zeta \subseteq W$  the Weyl group of $\mf l_\zeta$. The standard action and dot-action of $W_\zeta$ on $\mf h^\ast$ are defined in the same way. 

We fix a $W$-invariant non-degenerate symmetric bilinear form $(_-, _-)$ on $\mf h^\ast$. For any $\la\in \h^\ast$ and $\alpha\in \Phi^+$, we set $\langle \la,\alpha^\vee \rangle := 2(\la,\alpha)/(\alpha,\alpha)$. A weight $\la\in \h^\ast$ is called {\em integral } if $\langle \la, \alpha^\vee\rangle\in \Z$ for all $\alpha\in \Pi$. Denote by $\Lambda\subseteq \h^\ast$ the set of   integral weights.  A weight $\la \in \h^\ast$ is called {\em $W_\zeta$-anti-dominant} (resp. {\em $W_\zeta$-dominant}) if it is anti-dominant (resp. dominant) with respect to the dot-action of $W_\zeta$, that is, $\langle \la+\rho,\alpha^\vee\rangle \not \in \Z_{>0}$ (resp. $\langle \la+\rho,\alpha^\vee\rangle \not \in \Z_{<0}$) for positive roots $\alpha$ in $\mf l_\zeta$. A weight is said to be {\em anti-dominant} and {\em dominant} provided that it is $W$-anti-dominant and $W$-dominant, respectively. 

\subsubsection{}\label{sect::142}
%{\color{cyan} Following \cite{MS}, a finitely generated $\g$-module $M$ is called a {\em Whittaker module} provided that it is locally finite over the center $Z(\mf g)$  of $U(\g)$ and $U(\mf n_+)$.} {\color{red} explain some properties, like any Whittaker modules has finite length.}

Following notations in \cite{MS}, we denote by  $Y_\zeta(\la,\zeta)$ the family of Kostant’s simple Whittaker $\mf l_\zeta$-modules from \cite{Ko78},  where $\la\in \h^\ast$ and $\zeta\in \ch \mf n_+$. That is, $$Y_\zeta(\la,\zeta)= U(\mf l_\zeta)/\ker(\chi^{\mf l_\zeta}_\la)U(\mf l_\zeta)\otimes_{U(\mf n_+\cap \mf l_\zeta)}\C_\zeta,$$ 
where $\ker(\chi_\la^{\mf l_\zeta})$ is the kernel of the central character $\chi_\la^{\mf l_\zeta}: Z(\mf l_\zeta) \rightarrow \C$ associated with $\la$ and $\C_\zeta$ is the one-dimensional $\mf n_+\cap \mf l_\zeta$-module associated with $\zeta$. Let $\mf p_\zeta := \mf l_\zeta+\mf n_+$. Recall  the corresponding standard Whittaker module $M(\la,\zeta)$ from \cite{MS}:
\begin{align*}
&M(\la,\zeta) = U(\g)\otimes_{U(\mf p_\zeta)}Y_\zeta(\la,\zeta).
\end{align*}  
The top of $M(\la,\zeta)$ is simple, which we denote by $L(\la,\zeta)$. The  modules $\{L(\la,\zeta)|~\la\in \h^\ast, \zeta\in \ch \mf n_+\}$ constitute a complete set of simple Whittaker modules. We define $M(\la):=M(\la,0)$, which is the Verma module of highest weight $\la$. Also, we put $L(\la):=L(\la,0)$. For any $\la,\mu\in\h^\ast$ and $\zeta,\eta\in \ch \mf n_+$, we have 
\begin{align*}
&M(\la,\zeta)\cong M(\mu,\eta) \Leftrightarrow L(\la,\zeta)\cong L(\mu,\eta) \Leftrightarrow \zeta=\eta~\text{and}~\la\in W_\zeta\cdot \mu.
\end{align*}
This allows us to parametrize simple Whittaker modules by   $\zeta\in \ch \mf  n_+$  and  $W_\zeta$-anti-dominant weights  $\la\in \h^\ast$. We refer to \cite{Mc,MS} for
all details.

Following \cite{MMM22,Ma23} a $\g$-module $M$ will be called {\em Kostant positive} provided that the answer to Kostant’s problem for $M$ is positive, namely, the representation map $\phi: U(\g)\rightarrow \mc L(M,M)$ is surjective, and {\em Kostant negative} otherwise. Our first main result is the following (see also Theorem \ref{thm::maintmAintext}):
\begin{thmA}  \label{mainthm1}
		% 	If  $\la\in \nu+\Lambda$, where $\nu$ is a dominant weight such that $W_\nu=W_\zeta$. 
		Let $\zeta\in\ch\mf n_+$ and $\la\in \h^\ast$ be $W_\zeta$-anti-dominant. Assume that  $\langle \la,\alpha^\vee\rangle \in \Z$, for all $\alpha\in \Pi_\zeta$. Then we have 
		\begin{itemize}
			\item[(1)] $M(\la,\zeta)$ is Kostant positive.  
			\item[(2)]  $L(\la,\zeta)$ is Kostant positive if and only if $L(\la)$ is.
		\end{itemize} 
\end{thmA}
In the case $\zeta=0$, the assumption in Theorem A is trivially satisfied. In this case, Theorem A recovers the classical result about positive answers to Kostant's problem for Verma modules. In fact, our proof of Theorem A depends heavily on the Kostant positivity for this special case.  Another special case of Theorem A recovers \cite[Proposition 5.13]{MS} for standard Whittaker modules, where the Kostant positivity of (thick) standard Whittaker modules  has been established for any dominant weight $\la$ having $W_\zeta$ as its stabilizer subgroup under the Weyl group action.   

 If the assumption in Theorem A is dropped, that is, the case that $\la\in \h^\ast$ such that  $\langle \la,\alpha^\vee \rangle\not\in \Z$ for some $\alpha\in \Pi_\zeta$, then both conclusions in Parts (1) and (2) are no longer true in general; see, e.g., \cite[Section 2.3]{Jo2}.  
 
 \subsubsection{} \label{sect::143}
 For a given weight $\la \in \h^\ast$, denote by  $\chi_\la:= \chi_\la^\g:  Z(\g) \rightarrow \C$ the central character associated with $\la$.  Denote by $\g\mod$ the category of finitely-generated $\g$-modules and $\g\mod_{\chi_\la}$ its Serre subcategory generated by all simple modules that admit the central character $\chi_\la$. Let $\pr_{\chi_\la}(-):\g\mod \rightarrow \g\mod_{\chi_\la}$ be the natural projection.  Let, further, $\mc O_\la$ denote the indecomposable direct
 summand of the category $\mc O$ containing $L(\la)$.

Fundamental amongst all simple Whittaker modules are Kostant's simple Whittaker $\g$-modules  from \cite{Ko78}, which are exactly  $L(\la,\zeta)$ with $\zeta$ non-singular by definition. In this case, we have $M(\la,\zeta) = L(\la,\zeta)\in \g\mod_{\chi_\la}$ and $\la$ can be chosen to be anti-dominant. Our second main result is the following statement, which shows that the Kostant’s problem for these modules is equivalent to a homological problem about  translated simple modules. 
\begin{thmB}	Suppose that $\zeta\in \ch \mf n_+$ is non-singular. Let $\la\in \h^\ast$ be an (arbitrary) anti-dominant weight and $\chi:=\chi_\la$. Then the following are equivalent:
	\begin{itemize}
		\item[(1)] $L(\la,\zeta)$ is Kostant positive.   
		\item[(2)] Let $\theta: \g\mod_{\chi}\rightarrow \g\mod_{\chi}$ be a projective endo-functor, i.e., $\theta$ is a direct summand of $\pr_{\chi}(V\otimes-): \g\mod_{\chi}\rightarrow \g\mod_{\chi}$, for some finite-dimensional weight $\g$-module $V$. Then  $\theta L(\la)\in \mc O_\la$, and the lengths of the tops of  $\theta L(\la) $ and $ \theta L(\la,\zeta)$ are the same. 
	\end{itemize}
	\end{thmB}
 
Based on previous  results by various authors, we give a complete classification of Kostant positive simple Whittaker modules over $\mf{sl}(2)$ for the first time; see Corollary \ref{cor::13}. For type A Lie algebras, a family of negative examples of non-singular simple Whittaker modules can be found in Example  \ref{ex::nonKo}, including Joseph's $\mf{sl}(2)$-example from \cite[Section 2.3]{Jo2}.

 \subsubsection{} \label{sect::144} 
The final piece of motivation   for the present paper is the Kostant's problem for Whittaker modules over quasi-reductive Lie superalgebras. We are mainly interested in the following  quasi-reductive 	Lie superalgebras from Kac’s list \cite{Ka1}:
\begin{align*}
&\widetilde{\g} = \gl(m|n), \mf{osp}(2|2n),\text{ or } \pn.
\end{align*} In particular, these Lie superalgebras are the so-called {\em type I Lie superalgebras}, that is, each of them is equipped with a compatible $\Z$-grading $\widetilde{\g}=\widetilde{\g}_{-1}\oplus\widetilde{\g}_0\oplus \widetilde{\g}_1$, which we refer to as a {\em type-I} grading in the paper.    
In what follows, we let $$\g:=\widetilde{\g}_\oa,$$ and use the same notations and terminologies   defined for  structures and representations of $\g$ in the paper. 	Fix a triangular decomposition of $\widetilde{\g}$ in the sense of \cite{Ma}:
		\begin{align}
		&\widetilde{\mf g}=\widetilde{\mf n}_-\oplus \h \oplus \widetilde{\mf n}_+,\label{defn::triansup}
		\end{align} with purely even Cartan subalgebra $\h$ and nilradicals $\widetilde{\mf n}_{\pm}$, extending the triangular decomposition of ${\g}$ in   \eqref{defn::trian}, that is, $\widetilde{\mf n}_{\pm} = \mf n_\pm\oplus \widetilde{\mf g}_{\pm1}$.  Denote by $\widetilde{\mc O}$ the  BGG category of $\widetilde{\g}$-modules corresponding to the Borel subalgebra $\mf h+\widetilde{\mf n}_+$. % It  is well-known that the Weyl group of a  Lie superalgebra does not control the blocks of $\widetilde{\mc O}$ in general. 
	The blocks of the category $\widetilde{\mc O}$ of a basic classical Lie superalgebra that are controlled by the Weyl group are called {\em typical}; see also \cite{Go}. A weight is said to be typical provided that it is a highest weight of a simple module in a typical block. For the Lie superalgebra $\pn$, we refer to \cite{Se02} for more details about the typical weights; see also subsection \ref{sect::3222}.
	
  Rather
		than attempting to repeat the definitions and properties of Whittaker modules over $\widetilde{\g}$ here, we refer the reader to \cite[Section 3]{Ch21} and subsections \ref{sect::31}, \ref{sect::313} for more details. In particular, the construction of standard Whittaker modules affords a natural generalization to representations of $\widetilde{\g}$, which we shall denote by $\widetilde{M}(\la,\zeta)$, and they are indexed by $\zeta\in \ch \mf n_+$ and $W_\zeta$-anti-dominant weights $\la\in \h^\ast$.	We shall denote the (simple) top of  $\widetilde{M}(\la,\zeta)$  by $\widetilde{L}(\la,\zeta)$.  Again,  $\widetilde{M}(\la) := \widetilde{M}(\la,0)$ and $\widetilde{L}(\la) := \widetilde{L}(\la,0)$ coincide with the Verma module of highest weight $\la$ and its simple quotient, respectively.  For given two Whittaker modules $M,N$ over $\widetilde{\g}$, 	 the $\g$-bimodule $\mc L(M,N)$  is also a $\widetilde{\g}$-bimodule as $U(\widetilde{\g})$ is a finite extension of $U(\g)$.  The Kostant's problem for Whittaker modules over $\widetilde{\g}$ is formulated in the same way. Similarly, we define Kostant positive and negative Whittaker modules.

  Our third main result is the following; see   Theorems \ref{mainth3}, \ref{thm::13}, \ref{lem::16}, \ref{mainthm3} and Corollary \ref{cor::einThmC}.

\begin{thmC} \label{mainthm2} Let  $\widetilde \g$ be a quasi-reductive Lie superalgebra with $\zeta\in \ch\mf n_+$ and a  $W_\zeta$-anti-dominant weight  $\la\in \h^\ast$.  
		\begin{enumerate}
			\item[(1)] Suppose that $\widetilde \g$ is either $\gl(m|n)$ or $\mf{osp}(2|2n)$. Let $\widetilde \g\mod$ be the category of finitely-generated $\widetilde{\g}$-modules and   $K(\_):\g\mod\rightarrow\widetilde{\g}\mod$ be the Kac induction functor  (see subsection \ref{sect::333}). Then, for any $\g$-module $V$ that admits the central character $\chi_\la$, the $\widetilde \g$-module $K(V)$ is Kostant positive if and only if $\la$ is typical and $V$ is Kostant positive. In particular, the following are equivalent:
			\begin{itemize}
				\item[(a)] $\widetilde{M}(\la,\zeta)$ is Kostant positive.
				\item[(b)]   $M(\la,\zeta)$ is Kostant positive, and $\la$ is typical. 
			\end{itemize}
					\item[(2)] Suppose that $\widetilde{\g}$ is one of the Lie superalgebras $\gl(m|n),$ $\mf{osp}(2|2n)$ and $\pn$. Assume that    $\langle \la,\alpha^\vee\rangle \in \Z$, for any $\alpha\in \Pi_\zeta$. Then we have 
					\begin{itemize}
						\item[(c)] $\widetilde{M}(\la,\zeta)$ is Kostant positive if and only if $\la$ is typical.
							\item[(d)]  $\widetilde{L}(\la,\zeta)$ is Kostant positive if and only if  $\widetilde{L}(\la)$ is.	
								\item[(e)] If $\la$ is typical, then the answers to Kostant's problem for $\widetilde{L}(\la,\zeta)$, $\widetilde{L}(\la)$ and $L(\la)$ are the same.  
					\end{itemize}  
		\end{enumerate}
		%In case that $\langle \la+\rho,\alpha^\vee\rangle \in \Z$, for any $\alpha\in \Pi_\zeta$, the standard Whittaker module $\widetilde{M}(\la,\zeta)$ is Kostant positive if and only if $\la$ is typical.
\end{thmC}

 In the special case $\zeta=0$, the implication $(b)\Rightarrow(a)$ of Part $(1)$ in Theorem C recovers   \cite[Proposition 9.4]{Go}. However, in loc. cit. the positive answers have been established for all basic classical Lie superalgebras. %In the case when $\la$ is integral,   the answer to Kostant's problem for $\widetilde{M}(\la,\zeta)$ is determined by the typicality of $\la$.

%	As a consequence of Theorem C, we prove  that this is a necessary and sufficient  condition. 
%\begin{thmC}  Assume that  and  $\langle \la+\rho,\alpha^\vee\rangle \in \Z$, for any $\alpha\in \Pi_\zeta$. Then we have 	\begin{itemize} 		\item[(1)]  $\widetilde{L}(\la,\zeta)$ is Kostant positive if and only if $\widetilde{L}(\la)$ is.		\item[(2)] If $\la$ is typical, then the answers to Kostant's problems for $\widetilde{L}(\la,\zeta)$, $\widetilde{L}(\la)$ and $L(\la)$ are the same.   	\end{itemize} \end{thmC}

\subsection{Structure of the paper}
%The paper is organized as follows. %In Section \ref{Sect1}, we establish Theorems A and B. 
 Before giving complete proofs of the main results, we provide some background materials  %including  projectively presentable modules and Whittaker modules, 
 and develop several preparatory results in subsections \ref{sect::220} and \ref{sect::230}. %In particular, we relate our main results with two exact functors from the category $\mc O$ to the category of Whittaker modules. 
 The proof of Theorem A is established in subsection  \ref{sect::22}.  In subsection \ref{sect::24}, we focus on Kostant's problem for non-singular simple Whittaker modules and give a proof of Theorem B. %Section \ref{sect::Kosuper} is devoted to a proof of Theorem C. 
 We explain our setup of Lie superalgebras and their representations in subsection \ref{sect::311}.  In subsection \ref{sect::3222} one can  find detailed examples of     Lie superalgebras of type I. We develop the Kac functor tools in subsection \ref{sect::333}. Finally, in subsections \ref{sect::32} and \ref{sect::33} we bring all of the above together to prove Theorem C.

\vskip 0.3cm
{\bf  Acknowledgment}. The author is partially supported by National Science and Technology Council grants of the R.O.C. and further acknowledges support from  the National Center for Theoretical Sciences. The author would like to thank Shun-Jen Cheng and Volodymyr Mazorchuk for  interesting discussions and helpful comments.
\vskip 0.3cm

 %\section{Proof of Theorem A} 
 \section{Kostant's problem for Whittaker modules over Lie algebras} 
\label{Sect1}

  The   goal of this section is to complete the proofs of Theorem A and Theorem B.  From the point of view of the following canonical isomorphisms  (see, e.g., \cite[6.8]{Ja}) 
\begin{align}
&\Hom_\g(V, \mc L(M,N)^\ad)\cong \Hom_\g(M\otimes V,N)\cong \Hom_\g(M,N\otimes V^\ast), \label{eq::Jo68iso}
\end{align}  for any   $\g$-modules $M, N$ and finite-dimensional weight $\g$-module $V$, the conventional wisdom of solving Kostant's problem for a given $\g$-module $M$ %that admits a central character of $\g$ 
is to estimate the dimension $\dim \Hom_\g(M\otimes V, M)$. 
Our strategy is to relate the main results with two exact functors from  the  category $\mc O$ to the category of Whittaker modules to determine the dimensions for standard and simple Whittaker modules.

	Recall that a weight $\nu\in \h^\ast$ is said to be {\em dominant} provided that $\langle \la+\rho,\alpha^\vee \rangle\not\in \Z_{<0},$ for all  $\alpha\in \Phi^+$. For such a weight, we denote by $W_\nu$ the stabilizer of $\nu$ under the dot-action of the Weyl group $W$. Let $\Pi_\nu$ be the subset $\Pi$ of simple roots $\alpha$ such that $\langle \nu+\rho, \alpha^\vee\rangle =0$ (compare with $\Pi_\zeta$ from Section \ref{sect::themainresults}). Note that $W_\nu$ can be identified with the Weyl group of a Levi subalgebra generated by $\g^{\alpha}$ ($\pm\alpha\in \Pi_\nu$) and $\h$. Let, further,   $\Lambda(\nu)\subset  \mf h^\ast$ denote the set of all weights $\la\in \nu+\Lambda$ that is anti-dominant with respect to the dot-action of $W_\nu$.   We denote by    $\mc F$    the full subcategory of $\mc O$ consisting of  all finite-dimensional weight  $\g$-modules.   %For such a weight, we denote by $W_[\nu]$ its integral Weyl group, that is, $W_[\nu]$ is the subgroup of $W$ consisting of elements $w$ such that $w\la-\la$ lies in the $\Z$-span of $\Phi$ in $\h^\ast$; see, e.g., \cite[Section 3.4]{Hu08}. 

%In this section, we fix an arbitrary character $\zeta \in \mf n_+$. For a given dominant weight $\nu\in \h^\ast$ with $W_\nu=W_\zeta$, we let  $\Lambda(\nu)\subset  \mf h^\ast$ denote the set of all weights $\la\in \nu+\Lambda$ that is $W_\zeta$-anti-dominant.  

 \subsection{Projectively presentable modules in $\mc O$} \label{sect::220}  In this subsection we introduce a full subcategory of $\mc O$ and present a  preparatory result for the use in later sections. 
 
 Fix a dominant weight $\nu\in \h^\ast$. A projective module $Q\in \mc O$ is called {\em $\nu$-admissible} if each simple quotient of $Q$ has highest weight lying in $\Lambda(\nu)$. Following \cite[Section 4.1]{CCM}, we let $\mc O^{\nu\text{-pres}}$ denote the full subcategory of $\mc O$ consisting of modules $M$  that have a two step presentation of the form
 \begin{align*}
 	&Q_2\rightarrow Q_1\rightarrow M\rightarrow 0,
 \end{align*}
 where both $Q_1$ and $Q_2$ are $\nu$-admissible projective modules in $\mc O$. We refer to  \cite{MaSt04} for further details; see also \cite{KoMa21,MaPW20}.
 
 Let $\alpha\in \Pi$. A module $N\in \mc O$ is called {\em $\alpha$-free} (resp. {\em $\alpha$-finite})  if every non-zero root vector  $f_\alpha\in \g^{-\alpha}$ acts freely on $N$ (resp. acts locally finitely on $N$).   The following lemma will be useful later on.
 \begin{lem} \label{lem::55} Let $\nu$ be a dominant weight. 
 	Suppose that there are the following short exact sequences in $\mc O$
 	\begin{align}
 		&0\rightarrow C\rightarrow M\rightarrow X\rightarrow 0, \label{eq::22} \\
 		&0\rightarrow C'\rightarrow M'\rightarrow Y\rightarrow 0, \label{eq::23}
 	\end{align}  satisfying the following conditions:
 	\begin{itemize}
 		\item[(a)]  $M, M'\in\mc O^{\nu\text{-pres}}$.
 		\item[(b)] All composition factors of $C, C'$ are $\alpha$-finite, for any $\alpha\in \Pi_\nu$.
 		\item[(c)] Both tops and socles of $X, Y$ are $\alpha$-free, for any $\alpha\in \Pi_\nu$. 
 	\end{itemize}  
 	Then we have  $\Hom_\g(M, M')\cong \Hom_\g(X,Y),$ as vector spaces.
 	%Let $X,Y\in \mc O$ such that any simple quotient $L(\la)$ of $X$ or $Y$ satisfies that $\la \in \Lambda$.  
 \end{lem}
 \begin{proof}
 	By assumption in Part (a), there are  $\nu$-admissible projective modules $P, K, P', K'$ such that $M\cong P/K$ and $M'\cong P'/K'$.  Applying the functors $\Hom_\g(-,M')$, $\Hom_\g(-,Y)$ to the   exact sequence in  $P\rightarrow K\rightarrow M \rightarrow 0$  and applying the exact functors $\Hom_{\g}(P,-)$, $\Hom_{\g}(K,-)$ to the short exact sequence in \eqref{eq::23}, respectively, we obtain the following exact sequences:

 	\begin{align}
 		&\xymatrixcolsep{2pc} \xymatrix{
 			&& 0 \ar[d]  & 0\ar[d]   \\
 			&& \Hom_\g(K,C')\ar[d] & \Hom_\g(P,C') \ar[d]  \\
 			0 \ar[r] & \Hom_\g(M,M') \ar[r]   \ar@<-2pt>[d]  &  \Hom_\g(K,M') \ar[r]  \ar@<-2pt>[d]  & \Hom_\g(P, M') \ar[d]\\ 0 \ar[r] &
 			\Hom_\g(M,Y) \ar[r]    &   \Hom_\g(K,Y) \ar[r] \ar[d]& \Hom_\g(P,Y)\ar[d] \\
 			&&0&0 } \label{eq::dia414}
 	\end{align} It follows that  $\Hom_\g(M,M')\cong \Hom_{\g}(M,Y)$,  since $ \Hom_\g(K,C')=\Hom_\g(P,C')=0$ by condition (b).
 	
 	Next, we apply $\Hom_\g(-,Y)$ to the short exact sequence in \eqref{eq::22}, then we get an exact sequence
 	\begin{align*}
 		&0\rightarrow \Hom_\g(X,Y)\rightarrow \Hom_\g(M,Y)\rightarrow \Hom_\g(C,Y).
 	\end{align*}  We obtain that  $\Hom_\g(X,Y)\cong \Hom_\g(M,Y)$, since $\Hom_\g(C,Y)=0$ by condition (c). The conclusion follows.  
 \end{proof}

 \subsection{Whittaker category $\mc N$} \label{sect::230}
  Following \cite{MS}, we denote by $\mc N$ the category of all Whittaker modules over $\g$. The category $\mc N$ decomposes as $\mc N = \bigoplus_{\zeta\in \ch \mf n_+} \mc N(\zeta)$ with the full subcategories $\mc N(\zeta)$ consisting of Whittaker modules $M$ on which $x-\zeta(x)$ acts locally nilpotently, for any $x\in \mf n_+$. By \cite[Theorem 2.6]{MS}, any  module $M\in \mc N$ is of finite length. Furthermore, the category $\mc N$ is closed under tensoring with finite-dimensional weight modules over $\g$; see \cite[Sections 4,  5]{MS}. We fix a character $\zeta\in \ch \mf n_+$ throughout this subsection. %As a consequence, for a given $M\in \mc N$, it follows by \eqref{eq::Jo68iso} that $M$ is Kostant positive if and only if $$\dim\Hom_\g(M\otimes V, V) = \dim \Hom_\g(V, (U(\g)/\Ann_{U(\g)}M)^\ad),$$ for any $V\in \mc F$.

  \subsubsection{A Mili{\v{c}}i{\'c}-Soergel type functor} \label{sect::221} Let $\nu \in \h^\ast$ be a dominant weight such that $W_\nu =W_\zeta$. Consider the following  exact functor from \cite[Section 7.3.1]{CCM}: 
 \begin{align}
 	&T_{\nu,\zeta}(-):\mc L(M(\nu),-)\otimes_{U(\g)} M(\nu,\zeta): \mc O \rightarrow \mc N(\zeta).
 \end{align}  In loc. cit., this functor is introduced for more general Lie superalgebras. In our present setup of reductive Lie algebras, it is a combination of two classical functors, that is, the exact functor $\mc L(M(\nu), -)\otimes_{U(\g)}-: \mc   O\rightarrow \mc B_\nu$ studied in \cite[Kapitel 6]{Ja} and \cite{BG} and the exact functor $-\otimes_{U(\g)}M(\nu,\zeta): \mc B_\nu\rightarrow \mc N(\zeta)$ studied in \cite[Theorem 5.3]{MS}, where $\mc B_\nu$ is the category of Harish-Chandra $\g$-bimodules that admit the central character $\chi_\nu$ on right sides. The functor $T_{\nu,\zeta}$ restricts to a fully faithful functor from $\mc O^{\nu\text{-pres}}$ to  $\mc N(\zeta)$. Furthermore, $T_{\nu,\zeta}(M(\mu))\cong M(\mu,\zeta)$ for $\mu\in \nu+\Lambda$,  $T_{\nu,\zeta}(L(\mu))\cong L(\mu,\zeta)$ for $\mu\in \Lambda(\nu)$,  and $T_{\nu,\zeta}(L(\mu))=0$ otherwise; see \cite[Proposition 5.15]{MS}. In fact, if we let $\mc W(\zeta)$ denote the target category  $T_{\nu,\zeta}(\mc O)$, then $T_{\nu,\zeta}(-): \mc O \rightarrow \mc   W(\zeta)$ satisfies the universal property of a certain Serre quotient category (in the sense of Gabriel \cite{Gabriel}); see \cite[Theorem 37]{CCM}. The following is a consequence of Lemma \ref{lem::55}.
 \begin{cor} \label{cor::2}
 	Fix a dominant weight $\nu$ such that $W_\nu=W_\zeta$.
 	Let $X,Y\in \mc O$ such that their weight supports are contained in $\nu +\Lambda$. Suppose that the tops and socles of $X,Y$ are $\alpha$-free, for any $\alpha\in \Pi_\zeta$.  Then we have a vector space isomorphism  $  \Hom_\g(X,Y)\cong \Hom_\g(T_{\nu,\zeta}(X),T_{\nu,\zeta}(Y))$.
 \end{cor} 
 \begin{proof}
 	Let $P$ and $P'$ be the projective covers of $X$ and $Y$ in $\mc O$, respectively. We note that both $P$ and $P'$ are $\nu$-admissible. Let $K$ and $K'$ be the sum of all images of homomorphisms from $\nu$-admissible projective modules to the kernel of $P\rightarrow X$ and $P'\rightarrow Y$, respectively. Set $M:=P/K$ and $M':=P'/K'$. Then we get two short exact sequences in $\mc O$
 	\begin{align*}
 		&0\rightarrow C\rightarrow M\rightarrow X\rightarrow 0, \\
 		&0\rightarrow C'\rightarrow M'\rightarrow Y\rightarrow 0,  
 	\end{align*} that satisfy the conditions (a), (b) and (c) in Lemma \ref{lem::55}. As a consequence, we obtain 
 	\begin{align*}
 		&\Hom_\g(X,Y)\cong \Hom_\g(M,M')\cong \Hom_\g(T_{\nu,\zeta}(M),T_{\nu,\zeta}(M'))\cong \Hom_\g(T_{\nu,\zeta}(X),T_{\nu,\zeta}(Y)).
 	\end{align*} 
 \end{proof}
 %Put $M(\la,\zeta)\in \mc N(\zeta)$ and $L(\la,\zeta)$ to be the standard Whittaker module and its simple quotient.  

 \subsubsection{The Backelin functor $\Gamma_\zeta$}
In this subsection, we introduce the exact {\em Backelin's functor} $\Gamma_\zeta(-):\mc O\rightarrow \mc N(\zeta)$ from \cite[Section 3.1]{B}, which is closely related to the functor $T_{\nu,\zeta}(-)$ and of independent interest.  We refer to \cite{B} for more details. 

For a given module $M\in \mc O$, denote by $\ov M$  the completion of $M$ with respect to its weight spaces, that is, $\ov M= \prod_{\la\in \h^\ast} M_\la$, where $M_\la =\{m\in M|~hm=\la(h)m,\text{ for all }h\in \h\}$ are weight subspaces. Then $\Gamma_\zeta(M)$ is defined as the $\g$-submodule of $\ov M$ of all vectors in $\ov M$ on which $x-\zeta(x)$ acts locally nilpotently, for any $x\in \mf n_+$.
	
  %We recall the description of Backelin's functor on the  Verma and simple modules. 
    Recall that    $\mc F$  denotes the category of  finite-dimensional weight  $\g$-modules.  For $\la\in \h^\ast$, it follows by \cite[Proposition 3]{AB} (see also \cite[Proposition 2.3.4]{AB}) that  $$\Gamma_\zeta(M(\la)\otimes V) \cong M(\la,\zeta)\otimes V,~\text{for any }V\in \mc F.$$  In particular, we have the most important case  $\Gamma_\zeta(M(\la))\cong M(\la,\zeta)$ from   \cite[Proposition 6.9]{B}.  Furthermore, in loc. cit., it is shown that if $\la$ is $W_\zeta$-anti-dominant then $\Gamma_\zeta(L(\la)) =L(\la,\zeta)$; otherwise, $\Gamma_\zeta(L(\la))$ is zero.

 % For any weight  $\la\in \h^\ast$, we have $\Gamma_\zeta(M(\la))\cong M(\la,\zeta)$.  Furthermore, by \cite[Proposition 3]{AB} we have $$\Gamma_\zeta(M(\la)\otimes V) \cong M(\la,\zeta)\otimes V,$$ for any finite-dimensional $\g$-module $V$. 
 
% Let $\la$ be a $W_\zeta$-anti-dominant weight. From \cite[Proposition 6.9]{B}, it follows that  $\Gamma_\zeta(M(\la))\cong M(\la,\zeta)$ and $ \Gamma_\zeta(L(\la))\cong L(\la,\zeta)$. %We note that $\Ann_\g L(\la)\subseteq \Ann_\g L(\la,\zeta)$.   Furthermore, by \cite[Proposition 3]{AB} we have $$\Gamma_\zeta(M(\la)\otimes V) \cong M(\la,\zeta)\otimes V,$$ for any finite-dimensional $\g$-module $V$ and $\la \in \h^\ast$.

The following lemma, which relate Kostant's problem with homomorphisms induced by Backelin functors on Whittaker modules, will be very useful.
 \begin{lem} \label{lem::5}
 Let $\zeta\in \ch \mf n_+$.	Suppose that  $\la\in \mf h^\ast$ is a $W_\zeta$-anti-dominant weight. Then the Backelin functor induces the following  injections  
 	\begin{align}
 		&\Gamma_\zeta(-):\Hom_{\g}(M(\la)\otimes V,M(\la))\hookrightarrow \Hom_{\g}(M(\la,\zeta)\otimes V,M(\la,\zeta)), \label{eq::02} \\
 		&\Gamma_\zeta(-):\Hom_{\g}(L(\la)\otimes V,L(\la))\hookrightarrow \Hom_{\g}(L(\la,\zeta)\otimes V,L(\la,\zeta)), \label{eq::03}
 	\end{align}  for any finite-dimensional weight $\g$-module $V$. 
 Furthermore, $M(\la,\zeta)$ is Kostant positive if and only if the induced map    $\Gamma_\zeta(-)$ in \eqref{eq::02} is surjective.
 	%Furthermore, if $\langle \la,\alpha^\vee\rangle\in \Z$ for any $\alpha\in \Pi_\zeta$, then the monomorphism $\Gamma_\zeta(-)$ in \eqref{eq::02} is an isomorphism. 
 \end{lem} 
  \begin{proof} We shall only prove the injectivity of $\Gamma_\zeta(-)$ in \eqref{eq::02}. The proof of that for  $\Gamma_\zeta(-)$ in \eqref{eq::03} is similar and will be omitted.

 	%Let $\la$ be $W_\zeta$-anti-dominant. First, we show that \[\Gamma_\zeta(-):\Hom_{\g}(M(\la)\otimes E,M(\la))\rightarrow \Hom_{\g}(M(\la,\zeta)\otimes E,M(\la,\zeta)),\] is always an injection for any $E\in \mc F$. 	To see this, we observe
 	 From the canonical isomorphism \[\Hom_\g(M(\la)\otimes V, L(\mu))\cong\Hom_\g(M(\la), L(\mu)\otimes V^\ast), \text{ for any $\mu\in \h^\ast$ and $V\in \mc F$}, \] we may conclude that any simple quotient of $M(\la)\otimes V$ has $W_\zeta$-anti-dominant highest weight, since $\la$ is  $W_\zeta$-anti-dominant. Let $f\in \Hom_{\g}(M(\la)\otimes V,M(\la))$ be non-zero, for $V\in\mc F$. Then we have $\Gamma_\zeta(\text{Im}(f))\neq 0$. Consider the following short exact sequence in $\mc N(\zeta)$:
 	\begin{align*}
 		&0\rightarrow \Gamma_\zeta(\text{Ker}(f)) \rightarrow M(\la,\zeta)\otimes V\xrightarrow{\Gamma_\zeta(f)} \Gamma_\zeta(\text{Im}(f)) \rightarrow 0. 
 	\end{align*} Suppose that  $\Gamma_\zeta(f) =0$, then we get   $\Gamma_\zeta(\text{Im}(f))=0$, which is  a contradiction.

  By	\cite[Lemma 2.2]{MS},    we have  $\Ann_{U(\g)} M(\la,\zeta) =\Ann_{U(\g)}M(\la)$, which we denote by $J(\la)$. % Since both modules $(U(\g)/J(\la))^{\text{ad}}$ and $\mc L(M(\la,\zeta), M(\la,\zeta))^{\text{ad}}$ are direct sums of simple finite-dimensional $\g$-modules, each	occurring with a finite multiplicity, 
   It follows by the canonical isomorphisms in  \eqref{eq::Jo68iso} that  $M(\la,\zeta)$ is Kostant positive if and only if 
 	\begin{align*}
 		&\dim \Hom_\g(V, (U(\g)/J(\la))^{\text{ad}}) =\dim \Hom_\g(V, \mc L(M(\la,\zeta), M(\la,\zeta))^{\text{ad}}),~\text{for any $V\in \mc F$.}
 	\end{align*} As we have already
 	mentioned in Section \ref{sec1}, $M(\la)$ is Kostant positive, it follows that    $U(\g)/J(\la)$ and $ \mc L(M(\la),M(\la))$ are isomorphic as bimodules.  As a consequence, for any $V\in \mc F$ we have 
 	\begin{align*}
 		&\Hom_\g(V, (U(\g)/J(\la))^{\text{ad}}) \\ &\cong  \Hom_\g(V, \mc L(M(\la), M(\la))^{\text{ad}}) \\
 		&\cong  \Hom_\g(M(\la)\otimes V, M(\la)) \\
 		&\hookrightarrow  \Hom_\g(M(\la,\zeta)\otimes V, M(\la,\zeta))\\
 		&\cong \Hom_\g(V, \mc L(M(\la,\zeta), M(\la,\zeta))^{\text{ad}}).
 	\end{align*}  This completes the proof.  
 \end{proof}

 %\begin{rem} If we drop the condition that $\la \in \Lambda(\nu)$ for some dominant $\nu$ that satisfies $W_\nu =W_\zeta$, then the conclusion in Lemma \ref{lem::5} is no longer true. The following example is taken from \cite{Jo2}, given by Joseph.  \end{rem}
 
 \begin{rem}
  In the case that the integral Weyl group of $\la$ is the same as the group $W_\zeta$, we may fix a dominant weight $\nu\in \la+\Lambda$ with $W_\nu=W_\zeta$ and then construct the corresponding functor $T_{\nu,\zeta}$  from Section \ref{sect::221}. One can establish a completely analogous conclusion of Lemma \ref{lem::5} in terms of the functor $T_{\nu,\zeta}(-)$. %Indeed, the functors $\Gamma_\zeta$ and $T_{\nu,\zeta}$ are isomorphic by \cite[Corollary 38]{CCM}. 
\end{rem} 
 \begin{rem}
  As we have already mentioned, the Kostant's problem for a standard Whittaker module  might have a  negative answer in general, even though  the answer for a Verma module is always positive. Therefore, the induced maps in \eqref{eq::02}, \eqref{eq::03} are not surjective in general.  
\end{rem}

  \subsection{Proof of Theorem A} \label{sect::22}
 We are going to give a complete proof of Theorem A from subsection \ref{sect::142}. We present a more general result in the following theorem, because of its independent interest. 
 	\begin{thm} \label{thm::maintmAintext}
 		  Fix a character $\zeta\in \ch \mf n_+$ and a dominant weight $\nu\in \h^\ast$ with $W_\nu =W_\zeta$.  Suppose that $X\in \mc O$ with weight supports contained in $\nu+\Lambda$ and it satisfies the following conditions:
 		\begin{itemize}
 			\item[(a)]  Both top and socle of $X$ are $\alpha$-free, for any $\alpha\in\Pi_\zeta$.  
 			\item[(b)] $\emph{\Ann}_{U(\g)}(X)\supseteq \emph{\Ann}_{U(\g)}(\Gamma_\zeta(X))$.
 			\item[(c)] $T_{\nu,\zeta}(X) \cong \Gamma_\zeta(X)$.
 		\end{itemize} 
 	Then the answers to Kostant's problem for $X$ and $\Gamma_\zeta(X)$ are the same.
 	 
 In particular, if $\la\in \h^\ast$ is $W_\zeta$-anti-dominant such that  $\langle \la,\alpha^\vee\rangle \in \Z$ for all $\alpha\in \Pi_\zeta$, then $M(\la)$ and $L(\la)$ satisfy the conditions (a), (b) and (c) above. Furthermore,  $M(\la,\zeta)$ is Kostant positive, and both the answers to Kostant's problem for $L(\la,\zeta)$ and $L(\la)$ are the same. 
 	\end{thm}
 \begin{proof}
 Using an argument similar to the proof of Lemma \ref{lem::5}, we obtain the following short exact sequences 
	\begin{align}
	&\xymatrixcolsep{2pc} \xymatrix{
		& 0\ar[d] & 0 \ar[d]  \\
		0 \ar[r]     & U(\g)/\Ann_{U(\g)}(X) \ar[r]^{(1)}  \ar@<-2pt>[d]_{(3)}   & \mc L(X,X) \ar[d]^{(4)} \\0 \ar[r]    &  U(\g)/\Ann_{U(\g)}(\Gamma_\zeta(X)) \ar[r]^{(2)} & \mc L(\Gamma_\zeta(X), \Gamma_\zeta(X)) } \label{eq::8}
\end{align}  Here $(1)$ and $(2)$ are natural representation maps.  By assumption (b), the injection $(3)$ is bijective. Using adjunction,  both top and socle of $X\otimes V$ are $\alpha$-free, for any $\alpha\in \Pi_\zeta$ and  $V\in \mc F$. By Corollary \ref{cor::2} and Lemma \ref{lem::5}, the injection in $(4)$ is indeed a  bijection, due to assumption (c). The first conclusion follows. 

Next we prove that conditions (a), (b) and (c) are satisfied whenever $X=M(\la)\text{ or }L(\la)$. Condition (a) for both $M(\la)\text{ and  }L(\la)$ follows by definition. As we have already mentioned above, the fact that $\Ann_{U(\g)} M(\la,\zeta) =\Ann_{U(\g)}M(\la)$ follows from	\cite[Lemma 2.2]{MS}. The fact that  $\Ann_{U(\g)} L(\la) = \Ann_{U(\g)} L(\la,\zeta)$ follows from \cite[Section 5.2]{Ch212}; see also \cite[Theorem B]{Ch212} for the special cases associated to integral weights $\la$.  From \cite[Proposition 5.15]{MS}, we have $T_{\nu,\zeta}(M(\la))\cong M(\la,\zeta) \cong \Gamma_\zeta(M(\la))$ and $T_{\nu,\zeta}(L(\la))\cong L(\la,\zeta) \cong \Gamma_\zeta(L(\la))$. Consequently, we have already reduced the answers of Kostant's problem for $M(\la,\zeta)$ and $L(\la,\zeta)$ to that for $M(\la)$ and $L(\la)$, respectively. The answer for $M(\la)$ is always positive by \cite[Corollary 6.4]{Jo}. This completes the proof.
 \end{proof}

\begin{cor} \label{coro::11}
	Suppose that $\la\in \h^\ast$ is an integral weight and $\zeta\in \ch \mf n_+$ is non-singular. 
	Then $L(\la,\zeta)$ is Kostant positive.  
\end{cor}

\begin{rem}
It is  worth pointing out that the proof of Theorem A can also be completed by a more transparent point of view from \cite{CCM}. Indeed, the functors $T_{\nu,\zeta}(-)$ and $\Gamma_\zeta(-)$ are isomorphic when restricted to the full subcategory of $\mc O$ of modules having weight supports included in $\nu+\Lambda$; see \cite[Corollary 38]{CCM}. Thus, the images of Verma  modules under these functors are isomorphic. Similarly for simple modules. %Alternatively, since $\Gamma_\zeta$ and $T$ are isomorphic by \cite[Corollary 38]{CCM}, this also provides an alternative argument for the conclusion.
\end{rem}

 %	{\color{red}	\begin{proof}   First, we shall prove Part (1) in Theorem A. 	By assumption, we have that $\langle \la,\alpha^\vee\rangle \in \Z$, for all $\alpha\in \Pi_\zeta$.   	Since $M(\la,\zeta)=M(w\cdot\la,\zeta)$, for any $w\in W_\zeta$, we may set $\la$ to be  $W_\zeta$-anti-dominant. We shall apply Lemma \ref{lem::5} by showing that the map $\Gamma_\zeta(-)$ from \eqref{eq::02} is surjective. To see this, let $E\in \mc F$. Observe that the tops and socles of both $M(\la)$ and $M(\la)\otimes V$ are $\alpha$-free, for any $\alpha\in \Pi_\zeta$.  Recall that $T(M(\la)\otimes V)=M(\la,\zeta)\otimes V$ for any $\la \in \Lambda(\nu)$; see, e.g., \cite[Proposition 5.15]{MS}. It follows that  $$\dim\Hom_{\g}(M(\la)\otimes V,M(\la)) =\dim\Hom_{\g}(M(\la,\zeta)\otimes V,M(\la,\zeta)),$$ by Lemma \ref{lem::55}. Therefore the surjectivity of  $\Gamma_\zeta(-)$  in \eqref{eq::02} follows. Alternatively, since $\Gamma_\zeta$ and $T$ are isomorphic by \cite[Corollary 38]{CCM}, this also provides an alternative argument for the conclusion. 	This completes the proof of Theorem \ref{mainthm1}. 	 It remains to prove Part (2) in Theorem A.   That $\Ann_{U(\g)} L(\la) = \Ann_{U(\g)} L(\la,\zeta)$ from  and the induced map    $\Gamma_\zeta(-)$ in \eqref{eq::03} is surjective, then   $L(\la,\zeta)$ is Kostant positive if and only if $L(\la)$ is.  \end{proof} }

 %	\subsection{Kostant problem for non-singular simple Whittaker modules}
 	%\subsection{Kostant positivity of simple Whittaker module associated to non-singular character $\zeta$}
 	 	\subsection{Kostant problem for Kostant's non-singular simple  Whittaker modules} \label{sect::24}
 In this subsection, we study Kostant's problem for Kostant's (non-singular) Whittaker modules from \cite{Ko78}. We start by giving a complete proof of the Theorem B from Section \ref{sec1}. 	For  $M\in \mc N$, we denote by $\ell(M)$ the length of a composition series of $M$. Moreover, we let $\text{rad}(M)$ and $\text{top}(M)$  denote the radical and the top of $M$, respectively. 
 	%\begin{thm}	Suppose that $\g$ is reductive and $\zeta$ is non-singular. Let $\la\in \h^\ast$ and $\chi:=\chi_\la$. Then the following are equivalent:  		\begin{itemize}  			\item[(1)] $L(\la,\zeta)$ is Kostant positive.  			\item[(2)] Let $\la'\in W_{[\la]}\cdot \la$ be the unique anti-dominant weight. Then  $\theta L(\la')\in \mc O_\la$ and $\ell(\emph{top} \theta L(\la')) = \ell({\emph{top}} \theta L(\la',\zeta))$, for any projective functor $\theta: \g\mod_{\chi}\rightarrow \g\mod_{\chi}$.   		\end{itemize}  	\end{thm}

 	\begin{proof}[{\bf Proof of Theorem B from subsection \ref{sect::143}}]  %Without loss of generality, we may assume that $\la=\la'$. 
 		By \cite[Corollary 3.9]{Ko78}, we have that  $\Ann_{U(\g)}L(\la,\zeta) = \Ann_{U(\g)}L(\la)=\Ann_{U(\g)}M(\la)$, which is  generated by the kernel of the central character $\chi=\chi_\la$; see also \cite[Proposition 2.1]{MS}. It follows by Lemma \ref{lem::5} that  $L(\la,\zeta)$ is Kostant positive if and only if \begin{align*}
 			&\dim \Hom_\g(\theta L(\la), L(\la)) =\dim \Hom_\g(\theta L(\la,\zeta), L(\la,\zeta)),
 		\end{align*} for any projective endo-functor $\theta$ on $\g\mod_\chi$.
 		
 		Assume that $L(\la,\zeta)$ is Kostant positive. We first prove that $\theta L(\la)\in \mc O_\la$. Suppose on the contrary that $\theta L(\la) = U\oplus U'$ such that $U$ is the maximal submodule of $\theta L(\la)$ inside $\mc O_\la$ (i.e., $U$ is the image of $\theta L(\la)$ under the natural projection $\mc O\rightarrow\mc O_\la$) and $U'\not\in \mc O_\la$ is non-zero.  We may observe that 
 		\begin{align*}
 			&\Hom_\g(\theta L(\la), L(\la))\\
 			&\cong \Hom_\g(U, L(\la)) \\
 			&\hookrightarrow \Hom_\g(\Gamma_\zeta U, L(\la,\zeta)) \text{ by Lemma \ref{lem::5}}\\
 			&\hookrightarrow \Hom_\g(\Gamma_\zeta U, L(\la,\zeta))\oplus \Hom_\g(\Gamma_\zeta U', L(\la,\zeta))\\
 			&\cong  \Hom_\g(\theta L(\la,\zeta), L(\la,\zeta)), \text{~since $\Gamma_\zeta\theta \cong \theta\Gamma_\zeta$}.	\end{align*}  
 		Since $\theta L(\la)$ admits a Verma flag, so does $U'$.  Thus, we have    $\Gamma_\zeta U'\neq 0$. Since $L(\la,\zeta)$ is the unique simple Whittaker module admitting the central character $\chi$, we may conclude that $\Hom_\g(\Gamma_\zeta  U', L(\la,\zeta))\neq 0,$ contradicting to the assumption. Thus,  we have already shown that $\theta L(\la)\in \mc O_\la$. Furthermore, this   implies that $\text{top}(\theta L(\la))$ is a direct sum of  $L(\la)$. Applying the Backelin functor $\Gamma_\zeta(-)$ to the short exact sequence $0\rightarrow \rad\theta L(\la) \rightarrow \theta L(\la)\rightarrow \text{top}\theta L(\la)\rightarrow 0$, we get an inclusion   $$L(\la,\zeta)^{\oplus \ell(\text{top}\theta L(\la))}\hookrightarrow \text{top}(\theta L(\la,\zeta)).$$ %Since every simple direct summand of $\text{top}(\theta L(\la))$ has anti-dominant highest weight by adjunction, 
 	 Putting these things together, we conclude the following 
 		\begin{align}
 			&\dim \Hom_\g(\theta L(\la), L(\la)) \label{eq::310}\\ &=\ell(\text{top}(\theta L(\la))) \label{eq::311}\\
 			&\leq \ell(\text{top}(\theta L(\la,\zeta))) \label{eq::312}\\
 			&=\dim \Hom_\g(\theta L(\la,\zeta), L(\la,\zeta)). \label{eq::313}
 		\end{align} Consequently, we have $\ell(\text{top}(\theta L(\la))) =
 		\ell(\text{top}(\theta L(\la,\zeta))).$
 		
 		Conversely,  assume that the conditions in Part (2) hold. Note that, with the assumption that $\theta L(\la) \in \mc O_\la$, the equalities and inequality in  \eqref{eq::310}-\eqref{eq::313} remains valid by the same argument.  It follows that the dimensions
 		$\dim \Hom_\g(\theta L(\la), L(\la))$ and $\dim \Hom_\g(\theta L(\la,\zeta), L(\la,\zeta))$ are the same. Consequently, $L(\la,\zeta)$ is Kostant positive.  This completes the proof.
 	\end{proof}

 	In the following corollary,  we give a complete answer to Kostant's problem for
 	simple Whittaker modules over $\mf{sl}(2)$. %Recall that we denote by $\rho\in \h^\ast$ the half sum of the positive root. 
 	\begin{cor} \label{cor::13}
 		Suppose that $\g=\mf{sl}(2)$.  Let $\alpha$ be the unique simple root of $\g$.  Then the following are equivalent for any $\zeta\in \ch \mf n_+$ and $W_\zeta$-anti-dominant weight $\la\in \h^\ast$:
 		\begin{itemize}
 			\item[(1)] $L(\la,\zeta)$ is Kostant positive.
 			\item[(2)] Either $\zeta=0$, or $\la \neq  -\frac{k+2}{4}\alpha$ for some  positive odd integer $k$.
 		\end{itemize}  
 	\end{cor}
 	\begin{proof}
 		Simple highest weight $\mf{sl}(2)$-modules are known to be Kostant positive; see, e.g., \cite[6.9 (10)]{Ja} and \cite[Corollary 6.4]{Jo}. Therefore, we may  assume that $\zeta\neq 0$. If $\la$ is integral, then the Kostant positivity of  $L(\la,\zeta)$ follows from Corollary   \ref{coro::11}. %Suppose that $\zeta$ is non-singular. Let $\la\in \h^\ast$.  If $\la$ is integral, then $L(\la,\zeta)$ is Kostant positive by Lemma \ref{lem::5}. \frac{n}{2}\alpha
 		Now, suppose that $\la$ is non-integral. In what follows, we identify elements in $\mf h^\ast$ with elements in $\C$ via the bijection $\mu\mapsto \langle \mu, \alpha^\vee\rangle$, for $\mu\in \h^\ast$.  For any non-negative integer $n$, the  composition factors of $L(n)\otimes L (\la)$ are
 		\[ L(\la+n),~L(\la+n-2), \ldots,~L(\la-n).\] 
 		Thus,  $\theta L(\la)$ is a direct sum of simple highest weight modules of anti-dominant highest weights, for any projective functor $\theta$. Applying the Backelin functor $\Gamma_\zeta,$ it follows that $\theta L(\la,\zeta)$ is also  semisimple. We thus conclude that  $\ell(\theta L(\la,\zeta))=\ell(\theta (\Gamma_\zeta(L(\la))))=\ell( \Gamma_\zeta(\theta L(\la))) = \ell(\theta L(\la))$. By Theorem B, we reduce the problem of determining  Kostant negative $L(\la,\zeta)$ to the problem of determining anti-dominant and non-integral weights $\la$ for which the translated simple module $\theta L(\la)$ does not lie in $\mc O_\la$.

 	  Let $\chi= \chi_\la$ and $m\in \Z\backslash \{0\}$. Set $s=s_\alpha\in W$ to be the simple reflection associated with $\alpha$. Note that $L(\la+m)\in \mc O_\chi$ if and only if $\la +m=s\cdot\la$, which is equivalent to the condition that $\la = -\frac{m+2}{2}\in \frac{1}{2}+\Z$.  The conclusion follows by Theorem B.
 	\end{proof}
 	
 	\begin{rem} \label{rem::10}
A special case of  Corollary \ref{cor::13} recovers an   example of Joseph from \cite[Section 2.3]{Jo2}, where the Kostant negativity of the simple Whittaker $\mf{sl}(2)$-module $L(-\frac{3}{4}\alpha, \zeta)$ with $\zeta\neq  0$ has been proved.  Strongly inspired by this example, we give a sufficient condition for a negative answer in the following corollary from the point of view of Theorem B.	\end{rem}

 	\begin{cor} \label{cor::11}
 		 Let $\zeta\in \ch\mf n_+$ be  non-singular and $\la$ be anti-dominant. Suppose that there is $w\in W$ satisfying the following conditions:
 		 \begin{itemize}
 		 	\item[(1)] $\dim L(w\cdot \la-\la)<\infty$.
 		 	\item[(2)] $w\cdot\la -\la$ does not lie in the $\Z$-span of $\Phi$ in $\h^\ast$.
 		 \end{itemize} Then $L(\la,\zeta)$ is Kostant negative.
 	\end{cor} 		
 	\begin{proof} We are going to show that $\theta L(\la)\notin \mc O_\la$. Indeed,  note that $L(w\cdot \la)$ is a submodule of $L(w\cdot \la-\la)\otimes L(\la)$.  Therefore, we have $\pr_{\chi_\la}(L(w\cdot \la-\la)\otimes L(\la))\notin \mc O_\la$. The conclusion follows by Theorem B. 
 	\end{proof}
In the following discussion, we prove that there always exist negative examples of non-singular simple Whittaker modules for any type A Lie algebra. 
   
 \begin{ex}  \label{ex::nonKo} 
 	Consider the Lie algebra $\g=\mf{gl}(n)$.  	We   illustrate the utility of Corollary \ref{cor::11}  by providing a family of non-singular simple Whittaker $\g$-modules that are Kostant negative.  
 	
 	 We realize $\gl(n)$ as the space of $n\times n$ complex matrices, and the triangular decomposition $\g=\mf n_-\oplus \mf h\oplus \mf n_+$ is just the usual decomposition into the lower triangular, diagonal
 	 and upper triangular matrices.   
 	 It is standard to identify $\h^\ast$ with a  $n$-dimensional complex space  $E:=\text{span}_\C\{\vare_1,~\vare_2,\ldots, \vare_n\}$ such that $\Phi^+$ and $\Pi$ correspond to $\{\vare_i-\vare_j|~1\leq i< j\leq n\}$ and $\{\vare_i-\vare_{i+1}|~1\leq i< n\}$, respectively. The bilinear form $( \cdot ,\cdot )$  is determined by $( \vare_i, \vare_j) =\delta_{ij}$, for any $1\leq i,j\leq n$. In what follows, we identify elements in $W$ with elements in the  symmetry group $S_n$ of order $n!$ with the  natural action on $E$ via $\sigma \vare_{i} = \vare_{\sigma(i)}$, for $\sigma\in S_n$ and $1\leq i\leq n$. % In what follows, we identify the elements in $W$ with the elements in $S_n$. 
 	 
 	 Let $w\in W$ be the element  determined by 
 	 \[w\vare_{i} =\vare_{i+1},~\text{for $1\leq i<n$ and }w\vare_n =\vare_1.\]	 Without loss of generality, we may choose $\rho$ to be $\sum_{k=1}^n (n-k)\vare_k$. Consider the vector $\la\in E$ determined by 
 	 \begin{align*}
 	 	&\la+\rho = \sum_{k=1}^n (\frac{k-1}{n})\vare_k.
 	 	\end{align*} 	%With slight abuse	of notation, we denote by $\la$ its restriction to $\mf h$. 
  	We may note that $\la$ is anti-dominant, since $\langle \la+\rho, \alpha^\vee \rangle = -\frac{1}{n}$, for any simple root $\alpha$. 
  	
  Fix a non-singular character $\zeta \in \ch \mf n_+$. In the case that $n=2$, it is proved that $L(\la,\zeta)$ is Kostant negative; see  \cite[Section 2.3]{Jo2}.  We claim that $L(\la,\zeta)$ is always Kostant negative, in general. We are going to show that the conditions (1) and (2) in Corollary \ref{cor::11} for our choices of $w$ and $\la$ are fulfilled. To see this, we calculate
  	\begin{align*}
    &w\cdot \la = w(\la+\rho) -\rho= (\frac{n-1}{n}-(n-1))\vare_1 -\sum_{k=2}^n (\frac{k-2}{n} - (n-k))\vare_k. 
  	\end{align*}
 	Thus, we obtain 
 		\begin{align*}
 		&w\cdot \la -\la  \\ &=(\frac{n-1}{n}-(n-1))\vare_1 +\sum_{k=2}^n (\frac{k-2}{n} - (n-k))\vare_k - \sum_{k=1}^n (\frac{k-1}{n})\vare_k + \sum_{k=1}^n (n-k)\vare_k\\
 		&=\frac{n-1}{n}\vare_1 -\frac{1}{n}(\vare_2 +\vare_3+\cdots + \vare_n).
 	\end{align*} It turns out that $L(w\cdot\la-\la)$ is finite-dimensional and $w\cdot\la-\la$ does not lie in the root lattice in $\mf h^\ast$. Consequently, the answer to Kostant's problem for  $L(\la,\zeta)$ is negative. 
 	\end{ex}

\section{Kostant's problem for Whittaker modules over Lie superalgebras} \label{sect::Kosuper}

 In the remainder of the paper, we shall assume that $\widetilde \g$ is one of the type I Lie superalgebras $\widetilde{\g}$ from Kac’s list \cite{Ka1}:
	\begin{align}
		&(\text{Type } {\bf A}):~\gl(m|n),~\mf{sl}(m|n),~\mf{psl}(n|n), \label{eq::claA}  \\
		&(\text{Type } {\bf C}):~\mf{osp}(2|2n),\label{eq::claC} \\
		&(\text{Type } {\bf P}):~\mf{p}(n),~[\mf{p}(n),\mf{p}(n)]. \label{eq::claP}
	\end{align}  %Such Lie superalgebras are called {\em type I} Lie superalgebras. 
The Lie superalgebra $\widetilde{\g}$ admits a  type-I gradation  $\widetilde \g=\widetilde \g_{-1}\oplus \widetilde \g_0\oplus\widetilde \g_1.$  
	In this section, we study Kostant's problem for Whittaker modules over these Lie superalgebras and then give a complete proof of Theorem C.
	
	%For simplicity, we only consider the Lie superalgebras $\gl(m|n), \mf{osp}(2|2n)$ and $\pn$ in the paper. 

	\subsection{Generalities on Lie superalgebras} \label{sect::311}
	\subsubsection{Setup}  Recall that we put $\g:=\widetilde{\g}_\oa$ and fixed a triangular decomposition of $\widetilde{\g}$ in \eqref{defn::triansup} in subsection \ref{sect::144} with Cartan subalgebra $\h$ and nilradicals $\widetilde{\mf n}_{\pm}$, which is compatible with the triangular decomposition $\g =\mf n_-\oplus \mf h\oplus \mf n_+$ of $\g$ in \eqref{defn::trian}, that is, the even parts of $\widetilde{\mf n}_\pm$ are $\mf n_\pm$ and the odd parts of  $\widetilde{\mf n}_\pm$ are $\widetilde \g_{\pm 1}$, respectively. %We define the standard Borel subalgebra $\mf b:=\mf h\oplus \mf n$.  
Denote by $\widetilde{\mc O}$ the BGG category with respect to the Borel subalgebra $\mf h+\widetilde{\mf n}_+$.  Let $\widetilde{\mc F}$ denote the full subcategory of $\widetilde{\mc O}$ of all finite-dimensional weight $\widetilde \g$-modules.

 	Denote by $\widetilde{\g}\mod$ the category of finitely-generated $\widetilde{\g}$-modules. We have the usual restriction functor $\Res(-): \widetilde{\g}\mod\rightarrow \g\mod$. The left adjoint functor of  $\Res(-)$ is the  induction functor 
$\Ind(-):=U(\widetilde{\g})\otimes_{U(\g)}-: \mf g\mod \rightarrow\widetilde{\g}\mod$. For given two Whittaker modules $M,N$, recall that we define $\mc L(M,N)$ as the $\g$-bimodule $\mc L(\Res M,\Res N)$, which is also a $\widetilde \g$-bimodule.

	\subsubsection{Whittaker category $\widetilde{\mc N}$} \label{sect::31}  A $\widetilde{\g}$-module is called a  {\em Whittaker module} provided that it restricts to a Whittaker module over $\g$ with respect to the triangular decomposition in \eqref{defn::trian}.  %We review some basic properties of the category of Whittaker modules over $\widetilde{\g}$. 
	Denote by $\widetilde{\mc N}$ the category of Whittaker modules over $\widetilde{\g}$. For each $\zeta\in \ch \mf n_+$, we let $\widetilde{\mc N}(\zeta)$ denote the full subcategory of $\widetilde{\mc N}$ of objects that restricts to objects in $\mc N(\zeta)$. Similar to the category $\mc N$, the action of $\mf n_+$ decomposes the category $\widetilde{\mc N}$ into a direct sum of full subcategories $\widetilde{\mc N}(\zeta)$ over all $\zeta\in \ch \mf n_+$. Again, any object in $\widetilde{\mc  N}$ has a finite composition series  and  the category $\widetilde{\mc N}$ is closed under tensoring with objects from $\widetilde{\mc F}$. We refer to \cite{Ch21} for all details.
	
	For any $M\in \widetilde{\mc N}$ and $V\in \mc F$, we note that 
	\begin{align}
	&\Hom_{\g}(V, \mc L(\Res M, \Res M)^{\text{ad}})\cong 
	\Hom_{\widetilde{\g}}(M \otimes \Ind V, M), 
	\end{align} by \eqref{eq::Jo68iso} and the adjunction. Putting these properties together,  we conclude that $M$ is Kostant positive if and only if 
	\begin{align*}
	&\dim\Hom_{\widetilde{\g}}(M\otimes \Ind V, M) = \dim \Hom_{\widetilde{\g}}(\Ind V, (U(\widetilde \g)/\Ann_{U(\widetilde \g)}M)^\ad),
	\end{align*} for any $V\in \mc F$. We remark that the $\widetilde{\g}$-modules  isomorphic to direct summands of $\widetilde{\g}$-modules of the form $\Ind V$ ($V\in \mc F$) are exactly the  projective-injective modules in $\widetilde{\mc F}$. 
 	
	% {\color{red}  For any $M\in \mc N$, it follows by \eqref{eq::Jo68iso} that $M$ is Kostant positive if and only if $$\dim\Hom_\g(M\otimes V, V) = \dim \Hom_\g(V, (U(\g)/\Ann_{U(\g)}M)^\ad),$$ for any $V\in \mc F$.}

	\subsubsection{Standard and simple Whittaker modules} \label{sect::313}
	For given $\la\in \h^\ast$ and $\zeta\in \ch \mf n_+$, the corresponding standard Whittaker module $\widetilde{M}(\la,\zeta)$ is defined as follows 
	\begin{align*}
		&\widetilde{M}(\la,\zeta) = U(\widetilde{\g})\otimes_{U(\widetilde{\mf p}_\zeta)}Y_\zeta(\la,\zeta),
	\end{align*}  where $\widetilde{\mf p}_\zeta : = \mf l_\zeta+\widetilde{\mf n}_+$ is a parabolic subalgebra of $\widetilde{\g}$. Similar to the Lie algebra case, these standard Whittaker modules can be parameterized by $\zeta\in \ch \mf n_+$ and $W_\zeta$-anti-dominant weights $\la$, namely,  we still have, for any $\la,\mu\in \h^\ast$ and $\zeta,\eta\in \ch \mf n_+$, that
\begin{align}
&\widetilde{M}(\la,\zeta)\cong \widetilde{M}(\mu,\eta) \Leftrightarrow \widetilde{L}(\la,\zeta)\cong \widetilde{L}(\mu,\eta) \Leftrightarrow \zeta=\eta~\text{and}~\la\in W_\zeta\cdot \mu. \label{eq::19}
\end{align} Recall that the (simple) top of   $\widetilde{M}(\la,\zeta)$ is denoted by $\widetilde{L}(\la,\zeta)$. The simple modules  $\widetilde{L}(\la,\zeta)$, for $(\la,\zeta)\in \mf h^\ast\times \ch \mf n_+$ constitute all simple Whittaker modules over $\widetilde{\g}$.

	  For the theory of Whittaker modules over $\widetilde{\g}$, these standard Whittaker $\widetilde{\g}$-modules  play the role similar to that which the standard Whittaker $\g$-modules play for the study of Whittaker modules over $\g$, and they share many favorable properties. We refer to \cite{Ch21} for further details. 

\subsubsection{Super Backelin's functor $\widetilde \Gamma_\zeta$} \label{Sect::314}
We recall the super version of Backelin’s original functor $\widetilde{\Gamma}_\zeta: \widetilde{\mc O}\rightarrow \widetilde{\mc N}(\zeta)$ from \cite[Section 5.2]{Ch21}, which naturally extends Backelin's functor $\Gamma_\zeta$ from \cite{B}. By definition, for any $M\in \widetilde{\mc O}$ , the image $\widetilde{\Gamma}_\zeta(M)$ is defined as the Whittaker $\g$-module $\Gamma_\zeta(\Res M)$, which is also a Whittaker  $\widetilde{\g}$-module.  For any $\la\in \h^\ast$,  it follows by \cite[Theorem 20]{Ch21} that  
\begin{align*}
	&\widetilde{\Gamma}_\zeta(\widetilde{M}(\la))=\widetilde{M}(\la,\zeta),~\text{ and }\widetilde{\Gamma}_\zeta(\widetilde{L}(\la)) \cong \left\{ \begin{array}{ll} \widetilde{L}(\la, \zeta),\quad \text{if $\la$ is  $W_\zeta$-anti-dominant;} \\ 
		0, \quad\text{otherwise}.\end{array} \right.   
\end{align*} The category $\widetilde{\mc N}(\zeta)$ is closed under tensoring with modules in $\widetilde{\mc F}$. Furthermore,  for any $M\in \widetilde{\mc O}$ and  $E\in \widetilde{\mc F}$ we have 
\begin{align*}
&\widetilde{\Gamma}_\zeta(M\otimes E)\cong \widetilde{\Gamma}_\zeta(M)\otimes E.
\end{align*}

\subsubsection{The functor $\widetilde{T}_{\nu,\zeta}$} Let $\zeta\in \ch \mf n_+$ and $\nu\in \h^\ast$ be dominant such that $W_\zeta=W_\nu$, i.e., the stabilizer of $\nu$ is $W_\zeta$ under the dot-action of $W$. We recall the functor $T_{\nu,\zeta}$   introduced in Section \ref{sect::221}. As we have already mentioned, this functor affords a natural generalization to Whittaker modules over $\widetilde{\g}$, which we shall denote by $\widetilde{T}_{\nu,\zeta}$:
 \begin{align}
	&\widetilde T_{\nu,\zeta}(-):\mc L(M(\nu),-)\otimes_{U(\g)} M(\nu,\zeta): \widetilde{\mc O} \rightarrow \widetilde{\mc N}(\zeta).
\end{align} 
 By \cite[Proposition 33]{Ch21}, it follows that
 \begin{align*}
 &\widetilde{T}_{\nu,\zeta}(\widetilde{M}(\mu))\cong \widetilde{M}(\mu,\zeta), \text{ for any }\mu\in \nu+\Lambda; \\
 &\widetilde{T}_{\nu,\zeta}(\widetilde{L}(\mu))\cong \widetilde{L}(\mu,\zeta), \text{ for any }\mu\in \Lambda(\nu).
 \end{align*}

  The following is the main result in this subsection. 
	\begin{thm}  \label{mainth3} Let $\widetilde \g$ be a type I Lie superalgebra with $\zeta\in \ch \mf n_+$. 
		Suppose that $\la\in \mf h^\ast$ is $W_\zeta$-anti-dominant and   $\langle\la,\alpha^\vee\rangle \in \Z$, for any $\alpha\in \Pi_\zeta$. Then  we have 
		\begin{itemize}
			\item[(1)]  $\widetilde{M}(\la,\zeta)$ is Kostant positive if and only if $\widetilde{M}(\la)$ is.
			\item[(2)] $\widetilde{L}(\la,\zeta)$ is Kostant positive if and only if $\widetilde{L}(\la)$ is.
		\end{itemize} 
	\end{thm}
%\begin{comment}
	\begin{proof}
		The proof follows the same strategy as in the Lie algebra
		case given in Section \ref{Sect1}. Our goal is to establish some key steps, while omitting the parts that are analogous to the Lie algebra case and for which we refer to the proof of Theorem \ref{thm::maintmAintext} for details.
		
		Since $\widetilde{\g}$ is of  type I and $\la$ is $W_\zeta$-anti-dominant, it follows that tops of both $\widetilde{M}(\la)\otimes E$ and $\widetilde{L}(\la)\otimes E$ are $\alpha$-free, for any $E\in \widetilde{F}.$  %Let $X$ be $M(\la,\zeta)$ or $L(\la,\zeta)$, and let $\widetilde{X}$ be $\widetilde{M}(\la,\zeta)$ and be $\widetilde{L}(\la,\zeta)$ otherwise. 
		It follows by an argument similar to that used in the proof of Lemma \ref{lem::5} that the induced maps 
		\begin{align}
			&\widetilde\Gamma_\zeta(-):\Hom_{\widetilde\g}(\widetilde M(\la)\otimes E,\widetilde M(\la))\hookrightarrow \Hom_{\widetilde\g}(\widetilde M(\la,\zeta)\otimes E,\widetilde M(\la,\zeta)), \label{eq::021} \\
			&\widetilde{\Gamma}_\zeta(-):\Hom_{\widetilde\g}(\widetilde L(\la)\otimes E,\widetilde L(\la))\hookrightarrow \Hom_{\widetilde\g}(\widetilde L(\la,\zeta)\otimes E,\widetilde L(\la,\zeta)), \label{eq::031}
		\end{align} are injective, for any $E\in \widetilde{\mc F}$.  The fact that  they are surjective follows from an analogue of Corollary \ref{cor::2} using the functor $\widetilde{T}_{\nu,\zeta}$ instead of $T_{\nu,\zeta}$. Finally, the fact that $\Ann_{U(\widetilde{\g})}\widetilde{M}(\la,\zeta) =\Ann_{U(\widetilde{\g})}\widetilde{M}(\la)$ follows from \cite[Proposition 26]{Ch21} and the fact $\Ann_{U(\widetilde{\g})}\widetilde{L}(\la,\zeta) =\Ann_{U(\widetilde{\g})}\widetilde{L}(\la)$ follows from \cite[Section 5.2]{Ch212}. Mutatis mutandis the proof of Theorem \ref{thm::maintmAintext} and  the conclusion follows. 
\end{proof}

\subsection{Examples: Lie superalgebras $\gl(m|n)$, $\mf{osp}(2|2n)$ and $\pn$}\label{sect::3222}
In this subsection, we  set up the usual description of the distinguished triangular decompositions for $\widetilde \g =\mf{gl}(m|n), \mf{osp}(2|2n)$ and $\pn$.  We refer to \cite[Chapter 1]{ChWa12} for further details. %We emphasize that the study of simple   modules over Lie superalgebras from \eqref{eq::claA}-\eqref{eq::claP} can be reduce to that over Lie superalgebras $\gl(m|n)$, $\mf{osp}(2|2n)$ and $\pn$; see also \cite[Proposition 6.2]{CM}. 
 Recall that we define $\g := \widetilde{\g}_\oa$. 

\subsubsection{The general linear Lie superalgebra $\mf{gl}(m|n)$}  \label{sect::321}
For non-negative integers $m,n$, the general linear Lie superalgebra $\mathfrak{gl}(m|n)$ 
can be realized as the space of $(m+n) \times (m+n)$ complex matrices
\begin{align*}
	\widetilde{\g}:=\gl(m|n):=  	\left\{  \left( \begin{array}{cc} A & B\\
		C & D\\
	\end{array} \right) \|~ A\in \C^{m\times m}, B\in \C^{m\times n}, C\in \C^{n\times m}, D\in \C^{n\times n}
	\right\},
\end{align*}
%where $A,B,C$ and $D$ are $m\times m, m\times n, n\times m, n\times n$ matrices, respectively.
with the Lie bracket    given by the super commutator. The even subalgebra $\g$ of $\widetilde{\g}$ is isomorphic to $\gl(m)\oplus \gl(n)$.

Let $E_{ab}$, denote the $(a,b)$-matrix unit, for $1\leq a,b \leq m+n$. The standard Cartan subalgebra $\h := \text{span}_{\C}\{E_{aa}|~1\leq a\leq n\}$ consists of diagonal matrices.  Denote by $\{\vare_1, \vare_2,\ldots, \vare_{m+n}\}\subseteq \mf h^\ast$ the dual basis determined by $\vare_b(E_{aa}) =\delta_{ab}$. The set $\Phi_\oa^+$ of positive even roots and the set $\Phi_\ob^+$ of positive odd roots are given by 
\begin{align*}
	&\Phi^+_\oa:=\{\vare_a-\vare_b|~1\leq a< b \leq n, ~m+1\leq a <b\leq m+n \}.\\
	&\Phi_\ob^+:=\{\vare_a-\vare_b|~~1\leq a\leq m <  b \leq m+n\}.
\end{align*} 
We define the set $\Phi^-_\oa:=-\Phi^+_\oa$  of negative even roots  and  the set $\Phi^-_\ob:=-\Phi^+_\ob$  of negative odd roots.  We have the corresponding triangular decomposition $\widetilde \g=\widetilde{\mf n}_- \oplus \mf h\oplus \widetilde{\mf n}_+$, where $\mf n_\pm:=\bigoplus_{\alpha\in \Phi_\oa^\pm\cup \Phi^\pm_\ob}\widetilde \g^\alpha$, and the type I grading $\widetilde \g=\widetilde \g_{-1}\oplus \widetilde \g_0\oplus\widetilde \g_{1}$ is given by 
\begin{align}
&\widetilde \g_{-1}:=\bigoplus_{\alpha\in \Phi^-_\ob}\widetilde \g^\alpha,~\widetilde \g_0:=\widetilde{\g}_\oa,~\widetilde \g_{1}:=\bigoplus_{\alpha\in \Phi^+_\ob}\widetilde \g^\alpha. \label{defn::typeIgrad}
\end{align} Here  $\widetilde \g^\alpha$ denotes the root space corresponding to the root $\alpha$. 
The set $\Pi$ of simple even roots is given by 
\begin{align*}
&\Pi=\{\vare_a-\vare_{a+1}|~1\leq a<n,~n+1\leq a< m+n\}.
\end{align*}

The  bilinear form $({}_-,{}_-):\mf h^\ast \times \mf h^\ast \rightarrow \C$ is determined by $(\vare_a,\vare_b) = \delta_{ab}$ and $(\vare_{c},\vare_{d}) = -\delta_{cd}$, for $1\leq a,b\leq m,~n+1\leq c,d\leq m+n.$
Finally, we  refer  to \cite[Section 1.3]{ChWa12} and \cite[Section 2.2.6]{ChWa12}  for the definition of typical weights for $\gl(m|n)$.
  %We also  refer  to \cite[Section 1.3]{ChWa12} and \cite[Section 2.2.6]{ChWa12}   for more details about the Weyl vector $\widetilde{\rho}$ and typical weights. 

%{\color{red}  Recall that we  set $\widetilde \rho\in \h^\ast$ to be the Weyl vector, that is, it is the half sum of even positive roots minus the half sum of odd positive roots; see also \cite[Section 1.3]{ChWa12}. Note that a weight $\la = \sum_{i=1}^n\la_i\vare_i\in \h^\ast$ is typical  if and only if .}

\subsubsection{The orthosymplectic Lie superalgebra $\mf{osp}(2|2n)$}   \label{sect::322}
For a given positive integer $n$, the orthosymplectic Lie superalgebra $\mf{osp}(2|2n)$ is a   
 subsuperalgebra of $\mf{gl}(2|2n)$:
\begin{equation}
	\widetilde{\g}:=\mf{osp}(2\vert 2n)=
	\left\{ \left( \begin{array}{cccc} c &0 & x &y\\
		0 & -c& v & u\\
		-u^t& -y^t & A &B\\
		v^t &x^t & C& -A^t \\
	\end{array} \right)\|~
	\begin{array}{c}
		c\in \C;\,\, x,y,v,u\in \C^{n};\\
		A,B,C\in \C^{n^2};\\
		B, C \text{ are symmetric}.
	\end{array}
	\right\}. \label{eq::osp}
\end{equation}
We refer the reader to \cite[Section 1.1.3]{ChWa12} and \cite[Section 2.3]{Mu12} 
for more details.  The even subalgebra $\g$ of   $	\widetilde{\g}$  is isomorphic to 
$\C \oplus \mf{sp}(n)$.%, and the odd part of $\mf{osp}(2|2n)$ is isomorphic to $\C^2\otimes \C^{2n}$. 

Define 
$$H_\vare:=E_{11}-E_{22},~H_i:=E_{2+i,2+i}-E_{2+i+n,2+i+n},
\quad \text{ for }1\leq i\leq n.
$$
The standard Cartan subalgebra $\mf h$ of $\mf{osp}(2|2n)$ 
is spanned by $H_\vare$ and $H_1,H_2,~\ldots, H_n$. We let 
$\{\vare,\delta_1,\delta_2,\ldots,\delta_n\}$ be 
the basis of $\mf h^\ast$ dual to  $\{H_\vare,H_1,H_2,\ldots, H_n\}$.
The  bilinear form $({}_-,{}_-):\mf h^\ast \times \mf h^\ast \rightarrow \C$ is given by $$(\vare,\vare) =1,~(\delta_i,\delta_j) =-\delta_{i,j}, (\vare,\delta_i)=0, \text{ for $1\leq i,j\leq n$}.$$ Define the set $\Phi^+_\oa$ of even positive roots  and 
the set $\Phi^+_\ob$ of odd positive roots  by 
\begin{align*}
	&\Phi^+_\oa:=\{\delta_i \pm\delta_j, 2\delta_p|~1\leq i< j <n, ~1\leq p \leq n \}.\\
	&\Phi_\ob^+:=\{\vare \pm \delta_p|~1\leq p\leq n\}.
\end{align*} 
%This gives rise to a triangular decomposition  $\mf g=\mf n^+\oplus \mf h \oplus \mf n^-$,   %with Borel subalgebra $\mf b:= \mf b_\oa \oplus \mf g_1.$ 

  We define the corresponding triangular decomposition $\widetilde\g=\widetilde{\mf n}^+\oplus \mf h \oplus \widetilde{\mf n}^-$ and the  type I grading of $\mf{osp}(2|2n)$ in the same fashion as for $\gl(m|n)$. Again, we  refer  to \cite[Section 1.3]{ChWa12} and \cite[Section 2.2.6]{ChWa12}  for more details about the typical weights for $\mf{osp}(2|2n)$.
  %Again, we  refer  to \cite[Section 1.3]{ChWa12} and \cite[Section 2.2.6]{ChWa12}  for more details about the Weyl vector $\widetilde{\rho}$ and typical weights. 	
  
 %	{\color{red} A weight $\la\in \h^\ast$ is said to be {\em typical} if $\langle \la+\widetilde \rho,\alpha \rangle =0$, for some $\alpha\in \Phi_\ob$.}

%{\color{red} define the grading operator $D^{\gl(m|n)}:=\text{diag}(1,1,\ldots,1,\underbrace{0,\ldots,0}_{n}):=\sum_{i=1}^mE_{ii}$.}

\subsubsection{The periplectic Lie superalgebra} \label{subsect::pen} 
The {\em periplectic Lie superalgebra}~$\pn$  is a subalgebra of   $\mf{gl}(n|n)$.  
The standard matrix realisation is given
by
\[ \mf{pe}(n)=
\left\{ \left( \begin{array}{cc} A & B\\
	C & -A^t\\
\end{array} \right)\| ~ A,B,C\in \C^{n\times n},~\text{$B$ is symmetric and $C$ is skew-symmetric} \right\}.
\]We refer the reader to  \cite[Section 1.1]{ChWa12} for more details.

 Let $\widetilde{\g}:=\pn$. The even subalgebra $\g$ of $\widetilde{\g}$ is isomorphic to  $\gl(n)$. Let $e_{ij}: = E_{ij} -E_{n+j,n+i}\in \g$,  for all $1\leq i,j \leq n$.  The standard Cartan subalgebra $\mf h$ is defined as $\mf h: = \bigoplus_{1\leq i \leq n}\C e_{ii}$. %{\color{red} We can define the associated grading operator $D^{\pn}:=\sum_{i=1}^n e_{ii}\in Z(\g)$.   Just like in subsections \ref{sect::321}, \ref{sect::322}, all standard Whittaker modules and simple Whittaker modules are $\Z$-graded such that the homogeneous components are eigenspaces for $D^{\pn}$ with different eigenvalues;  see also \cite[Subsection 5.3]{CM} for details.}
Let $\{\vare_i|i=1,\ldots,n\}$ be the dual basis of~$\{e_{ii}\}_{1\leq i \leq n}$ in $\mathfrak{h}^*$.  The set $\Phi_\oa^+$ of positive even roots and the set of $\Phi_\ob^+$ positive odd roots are   given by 
\begin{align*}
	&\Phi^+_\oa:=\{\vare_i-\vare_j|~1\leq i< j <n\},
	&\Phi_\ob^+:=\{\vare_i+\vare_j|~1\leq i\leq j\leq n\}.
\end{align*}
The set of negative odd roots is given
\begin{align*}
	&\Phi_\ob^-:=\{-\vare_i-\vare_j|~1\leq i< j\leq n\}.
\end{align*} We may note that $\Phi_\ob^-\neq -\Phi_\ob^+$. 
The simple system $\Pi$ of $\g$ is given by 
\begin{align*}
	&\Pi = \{\vare_i-\vare_{i+1}|~1\leq i<n\}.
\end{align*}

Just like in subsections \ref{sect::321}, \ref{sect::322}, we define the corresponding triangular decomposition $\widetilde \g=\widetilde{\mf n}_+\oplus \mf h  \oplus \widetilde{\mf n}_-$ and the  type I grading of $\pn$ the same fashion. We note that $\dim \widetilde{\g}_{1} >\dim \widetilde{\g}_{-1}.$

The  bilinear form $({}_-,{}_-):\mf h^\ast \times \mf h^\ast \rightarrow \C$ is given by $(\vare_i,\vare_j) = \delta_{ij}$, for $1\leq i,j\leq n.$ Denote by $\Phi_\oa = \Phi^+_\oa\cup -\Phi^+_\oa$ set of all even roots. Recall that we denote $\rho=\frac{1}{2}\sum_{\alpha\in \Phi_\oa^+}\alpha$. A weight $\la\in \h^\ast$ is said to be {\em typical}
\cite[Section 5]{Se02} provided that   
	$(\la+\rho,\alpha) \neq 1$, for any $\alpha\in \Phi_\oa$.
 Furthermore, a weight $\la \in \h^\ast$  is said to be {\em weakly typical} \cite[Section 1.5]{CP} if 
$( \la+\rho,\alpha)\neq 1,$ for any $\alpha\in \Phi_\oa^+$.

   \subsection{Kac functor tools} \label{sect::333}
   We need some preparation before we can prove  Theorem C. Recall that we define $\g := \widetilde\g_{\oa}$.
   \subsubsection{Subalgebras $U(\widetilde \g_{\pm 1})$}
   	 %Let  $\Lambda(\widetilde \g_{-1}) \cong U(\widetilde \g_{-1})$ and $\Lambda(\widetilde \g_{1}) \cong U(\widetilde \g_{1})$ denote the exterior algebra of $\widetilde \g_{-1}$ and $\widetilde \g_{1}$ viewed as a subalgebra of $U(\widetilde \g)$, respectively.
   	 Note that the subalgebras $U(\widetilde \g_{-1})$ and $U(\widetilde \g_{1})$ of $U(\widetilde \g)$ can be  naturally viewed as the exterior algebras $\Lambda(\widetilde \g_{-1})$ and $\Lambda(\widetilde \g_{1})$ of $\widetilde \g_{-1}$ and $\widetilde \g_{1}$, respectively. For any non-negative integer $\ell$, we denote by $\Lambda^\ell (\widetilde \g_{\pm 1})$ the subspace spanned by all $\ell$-th exterior powers in $\Lambda (\widetilde \g_{\pm 1})$, respectively. We set 
   	\begin{align*}
   		&\Lambda^{\text{max}} (\widetilde \g_{-1}):=\Lambda^{\dim{\widetilde\g}_{-1}}(\widetilde \g_{-1}), ~\Lambda^{\text{max}} (\widetilde\g_{1}):=\Lambda^{\dim{\widetilde\g}_{1}} (\widetilde \g_{1}). \end{align*}
   	
   	The following lemma is taken from  \cite[Lemma 5.6]{CM}. 
   	
   	\begin{lem} \label{lem::CMlem56}
   		Assume that $\widetilde{\g}$ is either $\gl(m|n)$ or $\mf{osp}(2|2n)$.  		Let $X^{\pm}\in \Lambda^{\text{max}}(\widetilde \g_{\pm 1})$ be non-zero   vectors. Then we have 
   		\[X^+X^-= \Omega +\sum_{i} x_i^-r_ix_i^+,\]
   		where $\Omega\in Z(\g)$ $r_i\in U(\g)$, and $x_i^{\pm}\in \widetilde\g_{\pm 1}\Lambda(\widetilde\g_{\pm 1})$.
   	\end{lem} Let $\chi_\la: Z(\g)\rightarrow \C$ be the central character of $\g$ associated to $\la\in \h^\ast$. By \cite[Subsection
   	4.2]{Go}, we have $\chi_\la(\Omega) \neq 0$ if and only if $\la$ is typical.
   	% Next, we may observe that Recall that $J(\la):=\Ann M(\la,\zeta) =\Ann M(\la)$ (see, e.g., \cite[Lemma 2.2]{MS}), for any $\la\in \h^\ast$. By assumption, we have $U(\g)/J(\la) \cong \mc L(M(\la,\zeta),M(\la,\zeta)).$	Now, we claim that $\Res \mc L(M(\la,\zeta),M(\la,\zeta))$ is a weight module, and therefore it decomposes into a direct sum of simple $\g_\oa$-modules with finite multiplicities. First, since $\Res \mc L(M(\la,\zeta),M(\la,\zeta))$ is isomorphic to  $ \Lambda(\g_{-1})\otimes \Lambda(\mf g_{-1})^\ast\otimes \mc L(M_0(\la,\zeta),M_0(\la,\zeta))$, it follows that the restriction $\Res \mc L(M(\la,\zeta),M(\la,\zeta))$ is a direct sum of finite-dimensional simple $\g_\oa$-modules. Next, for any $V\in \mc F_0$, we have  $$\Hom_{\g_\oa}(V, \Res \mc  L(M(\la,\zeta),M(\la,\zeta))) \cong \Hom_{\g}(\Ind(V)\otimes M(\la,\zeta), M(\la,\zeta)).$$
 
 	\subsubsection{Kac modules}
 	%To prepare for a proof Theorem C for $\pn$, 
 	In this subsection we recall the Kac induction functor  from \cite[Section 2.4]{CM}, which originated in the work     of Kac \cite{Ka1, Ka2}. 
 	
 	Let $\widetilde \g_{\geq 0}:=\widetilde \g_0+\widetilde \g_1$. For any $V\in\g\mod$, we may extend $V$ to a $\widetilde\g_{\geq 0}$-module by letting $\widetilde\g_1V=0$ and define the {\em Kac module}: 
 	\begin{align}
 		&K(V):=U(\widetilde{\g})\otimes_{U(\widetilde{\g}_{\geq 0})}V.
 	\end{align}
 	This defines the so-called (exact) {\em Kac functor} $K(_-): \g\mod\rightarrow \widetilde{\g}\mod$. 	We may observe that  $$\Res K(V) \cong   \Lambda(\widetilde \g_{-1}) \otimes V,$$ 
 	as $\g$-modules. Since     $K(M(\la,\zeta))\cong \widetilde{M}(\la,\zeta),$  for any $\la\in \h^\ast$ and $\zeta\in \ch \mf n_+$, we can view the standard Whittaker $\g$-module $M(\la,\zeta)$ as the submodule $\Lambda^0(\widetilde{\g}_{-1})\otimes M(\la,\zeta)$ of $\Res \widetilde{M}(\la,\zeta)$. 
 	%For given $\la\in \h^\ast$ and $\zeta\in \ch \mf n_+$, we may note that $$\Res \widetilde{M}(\la,\zeta) \cong \Res U(\widetilde{\g})\otimes_{U(\widetilde \g_{\geq 0})} M(\la,\zeta)\cong   \Lambda(\widetilde \g_{-1}) \otimes M(\la,\zeta).$$ 

  For $\widetilde{\g} = \gl(m|n), \mf{osp}(2|2n)$ and $\pn$, 	we define the associated  grading operators:
 	\begin{align}
 		&D^{\widetilde{\g}}:=\left\{ \begin{array}{lll} \sum_{i=m+1}^{m+n}E_{ii},\quad \text{~for $\widetilde{\g}= \gl(m|n)$;} \\ 
 			H_\vare = E_{11}-E_{22},\quad \text{~~~for $\widetilde{\g}= \mf{osp}(2|2n)$;}\\
 		\sum_{i=1}^n e_{ii}= \sum_{i=1}^n(E_{ii}-E_{n+i,n+i}), \quad \text{~~for $\widetilde{\g} =\pn$.} \end{array} \right. \label{eq::graop}
 	\end{align} 
  Let $\chi: Z(\g)\rightarrow \C$ be a central character of $\g$. For any $V\in \g\mod_\chi$, the associated Kac module $K(V)$ and its quotient decomposes as a finite 
  direct sum of generalized eigenspaces of $D^{\widetilde \g}$, 
  namely,  % Consider the generalized eigenspaces decomposition of $K(V)$ respect to the grading operator $D^{\widetilde{\g}}$ from \eqref{eq::graop}:
  we have  \begin{align}
  	&K(V)=\bigoplus_{i=0}^{\dim \widetilde \g_{-1}} K(V)_{\chi(D^{\widetilde \g})-i_{\widetilde{\g}}} = \bigoplus_{i=0}^{\dim \widetilde \g_{-1}} \Lambda^i(\widetilde \g_{-1})\otimes V,\label{eq::30}
  \end{align} where 
  \begin{align*}
  	&	i_{\widetilde{\g}}:=  	    \begin{cases} i, &\mbox{ for~$\widetilde \g =\gl(m|n)$ or $\mf{osp}(2|2n)$;}\\  	    	2i,&\mbox{ for $\widetilde \g = \pn$}, \end{cases}
  \end{align*} and $K(V)_{\chi(D^{\widetilde \g})-i_{\widetilde{\g}}}$ denotes  the generalized eigenspace of $D^{\widetilde{\g}}$ with eigenvalue $\chi(D^{\widetilde \g})-i_{\widetilde{\g}}$, for $0\leq i\leq \dim \widetilde \g_{-1}$. 
  We may observe that $V\subset K(V)$ is the homogeneous component with the biggest generalized  eigenvalue.

  In the special case that $V$ is a simple $\g$-module, it admits a central character by Dixmier’s theorem \cite[Proposition 2.6.8]{Di}. 
  The decomposition in \eqref{eq::30} gives rise to $\Z$-gradings on $K(V)$ and its (simple) top such that the homogeneous components are eigenspaces for $D^{\widetilde \g}$ with different eigenvalues. We refer to \cite[Subsections 5.1, 5.2, 5.3]{CM} for all details. In particular, all standard   and simple Whittaker $\widetilde{\g}$-modules  are $\Z$-graded in this sense. 
 
 %  For any $\g$-module $V$ that admits a central character of $\g$, the associated Kac module $K(V)$ and its simple top can be decomposed into $D^{\widetilde \g}$-eigenspaces by Dixmier’s theorem \cite[Proposition 2.6.8]{Di}. This gives rise to $\Z$-gradings on these modules such that the homogeneous components are eigenspaces for $D^{\pn}$ with different eigenvalues. We refer to \cite[Subsections 5.1, 5.2, 5.3]{CM} for all details. In particular, all standard Whittaker $\widetilde{\g}$-modules and simple Whittaker $\widetilde{\g}$-modules  are $\Z$-graded in this sense. 

 	%We proceed to discuss Kostant's problem 
  	Denote by $\g\mod_{Z(\g)}$ the full subcategory of $\g\mod$ consisting of modules that are locally finite over $Z(\g)$. In the following propositions, we give  two sufficient conditions for Kostant negative Kac modules induced from objects in $\g\mod_{Z(\g)}$.
 	
 	\begin{prop} \label{prop::14}
 		Assume that $\widetilde \g=\gl(m|n)$ or $\mf{osp}(2|2n)$. Suppose that $V\in \g\mod_{Z(\g)}$. Then  $K(V)$ is Kostant negative unless $V\in \g\mod_{\chi_\la}$, for some typical weight $\la$.  %Let $V\in \g\mod_{\chi_\la}$, for some central character $\chi_\la$ of $\g$. Then  $K(V)$ is Kostant negative unless $\la$ is typical. %is typical, that is, $\chi=\chi_{\la}$ for some typical weight $\la$.
 	\end{prop}
 	\begin{proof} Since $V$ decomposes as a direct sum of modules in $\g\mod_\chi$ for some central characters $\chi$ of $\g$, we may assume without loss of generality that $V\in \g\mod_{\chi_\la}$ for some weight $\la\in \h^\ast$. Suppose that $K(V)$ is Kostant positive. 
 		Define $d:=\dim \widetilde{\g}_{-1}$. 
 		We fix a non-zero vector $X^-\in \Lambda^{\text{max}}(\widetilde \g_{-1})$. For each $f\in \mc  L(V, V)$, we define an associated linear map  $\hat f: K(V)\rightarrow K(V)$ by declaring, for any $y\in \Lambda^\ell (\widetilde \g_{-1})$ and  $v\in V$, 
 		\begin{align}
 			&\widehat f(y\otimes v) = \begin{cases}f(v), &\mbox{ if~$y=X^-$;}\\
 				0,&\mbox{if $\ell <d$}.
 			\end{cases}\label{eq::277}\end{align}
 		Let $f:=\text{Id}_V$ be the identity map on $V$.   By a direct calculation, $\C \widehat{\text{Id}_V}$ is a  $\g$-submodule of $\Hom_\C(K(V), K(V))^{\text{ad}}$, and so $\widehat{\text{Id}_V}\in \mc L(K(V), K(V)).$ By assumption, there exists $X\in U(\widetilde\g)$ such that $X(y\otimes v) = \widehat{\text{Id}_V}(y\otimes v)$, for any $y\in \Lambda^\ell(\widetilde \g_{-1})$ and  $v\in V$. %if we set $y$ above  to be a non-zero vector $X^-\in \Lambda^{\text{max}}(\widetilde \g_{-1})$,
 		This implies that   \begin{align}
 			&XX^-\otimes v = \widehat{\text{Id}_V}(X^-\otimes v) = \text{Id}_V(v)=v, \label{eq::3111}
 		\end{align}for any $v\in V$. 
 		 We write \begin{align}
 			&X = \sum_{i,j=0}^{d}  r_{ij} X^+_{i} X^-_{j}, \label{eq::322}
 		\end{align} so that   $X_i^+\in \Lambda^i(\widetilde \g_{1})$, $X_j^-\in \Lambda^j(\widetilde \g_{- 1})$ and $r_{ij}\in U(\g)$ for $0\leq i,j \leq d$. Then it follows by  \eqref{eq::3111} and \eqref{eq::30}  that $r_{d} X^+_{d}X^-v=v,$ for any $v\in V$. By Lemma \ref{lem::CMlem56}, there are $\Omega\in Z(\g)$ and $x_i^{\pm}\in \widetilde\g_{\pm 1}\Lambda(\widetilde\g_{\pm 1})$ such that  $X^+_dX^-= \Omega +\sum_{i} x_i^-r_ix_i^+$. By definition, there exists a non-zero vector $v\in V$ such that $(\chi_\la(z)-z)v=0$, for any $z\in Z(\g)$. It follows by  $0\neq v = r_{d} \Omega v =r_{d} \chi_\la(\Omega)v$ that $\chi_\la(\Omega)\neq 0$. This implies that $\la$ is typical. 
 	\end{proof}
 	
 	The following proposition is an analogue of Proposition \ref{prop::14} for $\pn$.

 %	The following lemma is taken from \cite[Proposition 4.1]{CP} (see also \cite[Lemma 3.1]{Se96}). \begin{lem} \label{lem::CP} Let $\la\in \h^\ast$. Then 		$\la$ is weakly typical if and only if the   Kac module
 		%\[K(L()):=U(\widetilde{\g})\otimes_{U(\widetilde{\g}_{\geq 0})}L(\la)\]	$K(L(\la))$ is simple.  	\end{lem} 
 	
 	\begin{prop}  \label{prop::16}
 		Assume that $\widetilde \g = \pn$. Let  $V\in \mc O$. Then $K(V)$ is Kostant negative unless $V\in \mc O_{\chi_\la}$, for some weakly typical weight $\la$.
 	\end{prop}
 	\begin{proof}  Without loss of generality, we may suppose that $V\in \mc O_{\chi_\la}$, for some $\la\in \h^\ast$. Assume that $K(V)$ is Kostant positive. Fix a non-zero vector $X^-\in \Lambda^{\text{max}}(\widetilde \g_{-1})$.
 		Let $\mc S$ be the full subcategory of 
 		$\g\mod$ consisting of objects $N$ in $\g\mod$ such that there exists a vector $X_N\in U(\widetilde \g)$ satisfying that $X_N(X^-\otimes v) =v$ in $K(N)$,  for any $v\in N$. We claim that $\mc S$ is closed under taking submodules and quotients. To see this, let $N\in \mc S$ and 
 		\begin{align*}
 	   &0\rightarrow N'\xrightarrow{\psi_1} N\xrightarrow{\psi_2} N'' \rightarrow 0,
 		\end{align*} be a short exact sequence in $\g\mod$. It induces an epimorphism $K(\psi_2) = 1\otimes \psi_2:  K(N)\rightarrow K(N'')$. By definition, we have  $$X_NX^-\otimes \psi_2(v) = K(\psi_2)(X_NX^-\otimes v)=\psi_2(v) \text{ in $K(N')$},$$ for any $v\in N$. This proves that $N''\in \mc S$. The assertion that $N'\in \mc S$ can be proved in a similar fashion. 
 	
 	   Finally, for  $\la\in \h^\ast$, the corresponding Kac module 
 	   $K(L(\la))$ is simple if and only if $\la$ is weakly typical, by \cite[Proposition 4.1]{CP}; see also \cite[Lemma 3.1]{Se02} and \cite[Lemma 5.11]{ChCo}. This completes the proof.
 	\end{proof}
 
 \begin{rem} 
 	In the case that  $\widetilde{\g} =\gl(m|n)$ or $\mf{osp}(2|2n)$, for a given simple $\g$-module $V$ the corresponding Kac module $K(V)$ is simple if and only if $V$ admits a central character $\chi_\la$ for some typical weight $\la$; see, e.g., \cite[Corollary 6.8]{CM} and also \cite{Ka1} and \cite[Lemma 3.3.1]{Ger98}. This together with the argument used in the proof of  Proposition \ref{prop::16}  provide an alternative proof of Proposition \ref{prop::14} for this  special case.
 \end{rem}

%\subsection{Kostant's problem for $M(\la,\zeta)$ over  $\gl(m|n)$}\label{sect::intro}
\subsection{Kostant's problem for   Whittaker modules over  $\gl(m|n)$ and $\mf{osp}(2|2n)$} \label{sect::32}
%In this section, we give prove Theorem C for Lie superalgebras $\gl(m|n)$ and $\mf{osp}(2|2n)$.

%Let $\g$ be a quasi-reductive Lie superalgebra. Let $\mc F$ be the category of all finite-dimensional $\g$-modules. Let $\mc P$ be the full subcategory of $\mc F$ of projective modules in $\mf F$. For each simple object $E\in \mc F$, let $P(E)$ denote the projective cover of $E$ in $\mc F$. 

\subsubsection{}%{Proof of Theorem C-(1) for $\widetilde \g = \gl(m|n), \mf{osp}(2|2n)$}  
In this subsection we establish the following result.

\begin{thm} \label{thm::13} Assume that $\widetilde \g$ is either $\gl(m|n)$ or $\mf{osp}(2|2n)$ with distinguished triangular decompositions   introduced in subsections \ref{sect::321}, \ref{sect::322}. 
	Then, for any $\zeta\in \ch\mf n_+, \la\in \h^\ast$, %be  a $W_\zeta$-anti-dominant weight. 
	 the following are equivalent:
	\begin{itemize}
		\item[(a)]  $\widetilde{M}(\la,\zeta)$ is Kostant positive.
		\item[(b)]   $M(\la,\zeta)$ is Kostant positive, and $\la$ is typical. 
	\end{itemize}
	In case that $\langle \la,\alpha^\vee\rangle \in \Z$, for any $\alpha\in \Pi_\zeta$, the Kostant's problem has a positive answer for $\widetilde M(\la,\zeta)$ if and only if $\la$ is typical.
\end{thm} 

The proof of Theorem \ref{thm::13} will be explained later. We proceed to formulate several consequences. 
In the extreme case when $\zeta=0$, the standard Whittaker module $\widetilde{M}(\la,0)$ coincides with the Verma module $\widetilde{M}(\la)$.  The following corollary gives a complete answer to Kostant's problem for Verma modules in terms of the typicality of their highest weights. The sufficient condition has already been established by Gorelik in \cite[Proposition 9.4]{Go} for all basic classical Lie superalgebras.
\begin{cor} \label{cor::14}
	Let $\la\in \h^\ast$.  The Kostant's problem has a positive answer for the Verma module $\widetilde M(\la)$ if and only if $\la$ is  typical. 
\end{cor}
\begin{proof} In case that $\zeta=0$, the   subgroup  $W_\zeta$ is trivial. It follows that every weight  in $\h^\ast$ is $W_\zeta$-anti-dominant. Since the Verma module  $M(\la)$ over $\g$ is always Kostant positive,  the conclusion follows by   Theorem \ref{thm::13}. 
\end{proof}

The analogue of the statement in Corollary \ref{cor::14} remains true for standard Whittaker modules $\widetilde{M}(\la,\zeta)$ with integral weights $\la$. 
\begin{cor} \label{cor::15}
	Let $\la\in \h^\ast$ be an integral weight. Then, for any $\zeta\in \ch \mf n_+$, the Kostant's problem has a positive answer for the standard Whittaker module $\widetilde M(\la,\zeta)$ if and only if $\la$ is  typical.  
\end{cor}
\begin{proof} Without loss of generality, we may choose $\la$ to be $W_\zeta$-anti-dominant by  \eqref{eq::19}. The conclusion is a consequence of Theorems A and \ref{thm::13}.
	%Let $\nu$ be an dominant weight such that $W_\nu=W_\zeta$ and $\la\in \nu+\Lambda$.  Theorem \ref{mainthm1} ensures that The Kostant's problem has a positive answer for   $M_0(\la,\zeta)$. The conclusion follows from  Theorem \ref{mainthm2}. 
\end{proof}

\subsubsection{} We are going to prove  Theorem \ref{thm::13} by establishing the following general result.
%We divide the proof of Theorem \ref{thm::13} into a few lemmas below. % module $V\g\mod_{Z(\g)}$ is said to admit a central character $\chi: Z(\g) \rightarrow\C$ provided that $\chi(z)v =zv,$ for any $z\in Z(\g)$ and $v\in V$.

\begin{thm} \label{lem::16} Assume that $\widetilde{\g}=\gl(m|n)$ or $\mf{osp}(2|2n)$.   Let $V$ be a  $\g$-module admitting a central character $\chi_\la: Z(\g)\rightarrow \C$ {\em (}i.e., $zv =\chi_\la(z)v,$ for any $z\in Z(\g)$ and $v\in V${\em )}, for some $\la\in \h^\ast$. Then the following are equivalent: 
	\begin{itemize}
		\item[(1)] $K(V)$ is  Kostant positive.
		\item[(2)] $V$ is   Kostant positive,  and $\la$ is typical.
	\end{itemize}	%Let $\zeta\in \ch\mf n_+$ and $\la\in \h^\ast$.  Suppose that $\widetilde M(\la,\zeta)$ is Kostant positive.  Then  $M(\la,\zeta)$ is also Kostant positive,  and $\la$ is typical. 
\end{thm} 
\begin{proof} %{\color{cyan} Define $d:=\dim \widetilde{\g}_{-1}$.  	 We fix a non-zero vector $X^-\in \Lambda^{\text{max}}(\widetilde \g_{-1})$. For each $f\in \mc  L(M(\la,\zeta), M(\la,\zeta))$, we define an associated linear map  $\hat f: \widetilde M(\la,\zeta)\rightarrow \widetilde M(\la,\zeta)$ by declaring, for any $y\in \Lambda^\ell (\widetilde \g_{-1})$ and  $m\in M(\la,\zeta)$, 	\begin{align} &\hat f(y\otimes m) = \begin{cases}f(m), &\mbox{ if~$y=X^-$;}\\ 		0,&\mbox{if $\ell <d$}.	\end{cases}\label{eq::27}\end{align} Let $f:=\text{Id}$ be the identity map on $M(\la,\zeta)$.   By a direct calculation, $\C \widehat{\text{Id}}$ is a  $\g$-submodule of $\Hom_\C(\widetilde M(\la,\zeta), \widetilde M(\la,\zeta))^{\text{ad}}$, and so $\widehat{\text{Id}}\in \mc L(\widetilde M(\la,\zeta), \widetilde M(\la,\zeta)).$ By assumption, there exists $X\in U(\widetilde\g)$ such that $X(y\otimes m) = \widehat{\text{Id}}(y\otimes m)$, for any $y\in \Lambda^\ell(\widetilde \g_{-1})$ and  $m\in M(\la,\zeta)$. %if we set $y$ above  to be a non-zero vector $X^-\in \Lambda^{\text{max}}(\widetilde \g_{-1})$,  This implies that   \begin{align} 		&XX^-\otimes m = \widehat{\text{Id}}(X^-\otimes m) = \text{Id}(m)=m, \label{eq::31} 	\end{align}for any $m\in M(\la,\zeta)$. 
		
		%Recall that $\widetilde M(\la,\zeta)$ is a $\Z$-graded module with respect to the grading operator $D^{\widetilde{\g}}$ from \eqref{eq::graop} and $M(\la,\zeta)$ is the homogeneous component with the biggest eigenvalue. We write \begin{align}	&X = \sum_{\ell=1}^{d}  r_\ell X^+_\ell X^-_\ell, \label{eq::32} 	\end{align} so that  $X_\ell^\pm\in \Lambda^\ell(\widetilde \g_{\pm 1})$ and $r_\ell\in U(\g)$ for $1\leq \ell \leq d$. Then it follows by  \eqref{eq::31}  that $$r_{d} X^+_{d}X^-m=m,$$ for any $m\in M(\la,\zeta)$. By Lemma \ref{lem::CMlem56}, there are $\Omega\in Z(\g)$ and $x_i^{\pm}\in \widetilde\g_{\pm 1}\Lambda(\widetilde\g_{\pm 1})$ such that  $X^+_dX^-= \Omega +\sum_{i} x_i^-r_ix_i^+$. It follows that  $m = r_{d} \Omega m =r_{d} \chi_\la(\Omega)m$, for any $m\in M(\la,\zeta)$. This implies that $\la$ is typical. }

		%{\color{cyan} 
		We  first prove the implication $(1)\Rightarrow (2)$. Suppose that $K(V)$ is Kostant positive.   By Proposition \ref{prop::14}, $\la$ is  a typical weight. 
		Let $f\in \mc  L(V, V)$. Recall that we defined an associated linear map $\hat f:   K(V)\rightarrow K(V)$ from \eqref{eq::277}. 		By a direct calculation, we find that the dimension of the cyclic $\widetilde{\g}$-submodule $U(\widetilde \g)\cdot \hat f\subset \Hom_\C(K(V),K(V))^\ad$ generated by $\hat f$ is at most $\dim U(\g)\cdot f$, under the adjoint action of $\g$. %Note that $\Hom_\C(M,M)^\ad$ is semisimple over $Z(\g)$, for any Whittaker module $M$ over $\widetilde{\g}$. 
	 This implies that $\hat f \in \mc L(K(V),K(V))$. By assumption, there exists $X\in U(\widetilde\g)$, depending on $f$, such that  $\hat f(m) = X(m)$, for any $m\in K(V)$. Again, if we write $X$ as the expression in \eqref{eq::322} with the fixed non-zero vector $X^-\in \Lambda^{\max}(\widetilde \g_{-1})$, then we get $r_{d}X^+_{d}X^-v = f(v)$ (here $d=\dim \widetilde \g_{-1}$), for any $v \in V$. We obtain that $$f(v)  =r_{d} \chi_\la(\Omega)v,$$ for any $v\in V$. Consequently, $f$ is the image of $r_d\chi_\la(\Omega)$ under the representation map $U(\g)\rightarrow \mc L(V, V)$. This proves that $V$ is Kostant positive. 
	
	Next, we prove the implication   $(2)\Rightarrow (1)$. Suppose that $\la$ is typical and $V$ is Kostant positive. 
  There is the following  canonical isomorphism 
	\[\psi: \Lambda(\widetilde \g_{-1}) \otimes \mc L(V, V) \otimes \Lambda(\widetilde \g_{-1})^\ast \cong \mc L(K(V), K(V)).\]
	Let $\phi: U(\widetilde \g)\rightarrow \mc L(K(V),K(V))$ be the natural representation map.  We note that there is a basis  $\{y_0,\ldots, y_\ell\}$  for $\Lambda(\widetilde \g_{-1})$  such that $$y_{i} y_j = \begin{cases}X^-, &\mbox{ for~$j=\ell-i$;}\\
		0,&\mbox{otherwise;}
	\end{cases}$$ for any $0\leq i\leq \ell$.
	
	Let $\{y_0^\ast,\ldots, y_\ell^\ast\} \subset \Lambda(\widetilde \g_{-1})^\ast$ be the dual basis with respect to $\{y_0,\ldots, y_\ell\}$. Let $f\in  \mc L(V, V)$.   By assumption, there is $r\in U(\g)$ such that $f(v)=rv,$ for any $v\in V$.  Set $X^+\in \Lambda^{\max}(\widetilde \g_1)$ to be a non-zero vector. By Lemma \ref{lem::CMlem56}, we have  $X^+X^-= \Omega +\sum_{i} x_i^-r_ix_i^+,$
	for some $\Omega\in Z(\g)$, $r_i\in U(\g)$ and $x_i^{\pm}\in \widetilde\g_{\pm 1}\Lambda(\widetilde\g_{\pm 1})$. We note that   
	$\chi_\la(\Omega)^{-1} \phi(xrX^+ y_{\ell-i})=x\otimes f\otimes y_i^\ast$, for any $0\leq i\leq \ell$ and $x\in \Lambda(\widetilde \g_{-1})$.  Consequently,  $\phi$ is an isomorphism. This completes the proof.
\end{proof}

\begin{proof}[Proof of Theorem \ref{thm::13}]
The fact that  $M(\la,\zeta)\in \mc O_{\chi_\la}$ follows from \cite[Lemma 2.2]{MS}. Therefore, the conclusion follows by Theorem A and  Theorem \ref{lem::16}.
\end{proof}

\subsection{Kostant's problem for  Whittaker modules over  $\pn$} \label{sect::33}

In this subsection, we study Kostant's problem for standard and simple Whittaker modules over $\pn$.  For each simple root $\alpha\in \Pi$, we set $s_\alpha\in W$ to be the simple reflection associated to $\alpha$.

% \subsubsection{Kostant's problem for standard Whittaker modules over $\pn$}
\begin{lem}\label{lem::18} Suppose that $\alpha\in \Pi$ such that either $\langle \la+\rho, \alpha^\vee\rangle \in \Z_{<0}$ or $\langle \la+\rho, \alpha^\vee\rangle \notin \Z$. If the Kostant's problem for $\widetilde M(\la)$ has an positive answer, then so does $\widetilde M(s_\alpha\cdot \la)$. 
\end{lem}
\begin{proof} %Let   $\widetilde M(\la)$ be  Kostant positive.  
	First, we suppose that $\langle \la+\rho,\alpha^\vee\rangle \in \mathbb Z_{<0}$. Let $C(\la):=\widetilde M(s_\alpha\cdot \la)/\widetilde M(\la)$.  Note that $C(\la)\cong K(M(s_\alpha\cdot \la)/M(\la))$ is $\alpha$-finite. Applying the bi-functor $\mc L(-,-)$ to the short exact sequence
	\begin{align*}
		&0\rightarrow \widetilde M(\la)\rightarrow \widetilde M(s_\alpha\cdot \la)\rightarrow C(\la)\rightarrow 0,
	\end{align*} we obtain the following commutative diagram with exact rows and columns: 
	\begin{align*}
		&\xymatrixcolsep{2pc} \xymatrix{
			&0\ar[d]& 0 \ar[d]  & 0\ar[d]   \\
			0\ar[r]&\mc L(C(\la),\widetilde M(\la))\ar[d]\ar[r]& \mc L(C(\la),\widetilde M(s_\alpha\cdot \la))\ar[d]\ar[r] & \mc L(C(\la),C(\la)) \ar[d]  \\
			0 \ar[r] & \mc L(\widetilde M(s_\alpha\cdot \la),\widetilde M(\la)) \ar[r]   \ar@<-2pt>[d]  & \mc L(\widetilde M(s_\alpha\cdot \la),\widetilde M(s_\alpha\cdot \la)) \ar[r]  \ar@<-2pt>[d]  & \mc L(\widetilde M(s_\alpha\cdot \la), C(\la)) \ar[d]\\ 0 \ar[r] &
			\mc L(\widetilde M(\la),\widetilde M(\la)) \ar[r]    &   \mc L(\widetilde M(\la),\widetilde M(s_\alpha\cdot \la)) \ar[r] & \mc L(\widetilde M(\la),C(\la)) } \label{eq::inLem18}
	\end{align*}   Since  both the top of $\widetilde M(\la)$ and the socle of $\widetilde M(s_\alpha\cdot \la)$ are $\alpha$-free, we conclude that 
	\[\Hom_{\widetilde{\g}}(\widetilde M(\la), E\otimes C(\la)) =0= \Hom_{\widetilde\g}(C(\la)\otimes E, \widetilde M(s_\alpha\cdot\la)),\]
	for any $E\in \widetilde{\mc F}$. Therefore, we  get that $\mc L(C(\la),\widetilde M(s_\alpha\cdot \la))=\mc L(\widetilde M(\la),C(\la))=0$ and so $\mc L(\widetilde M(s_\alpha\cdot\la),\widetilde M(s_\alpha\cdot\la)) \hookrightarrow L(\widetilde M(\la),\widetilde M(\la))$. We   observe that $\Ann_{U(\widetilde \g)}(\widetilde M(\la)) = \Ann_{U(\widetilde \g)}(\widetilde M(s_\alpha\cdot \la))$ (see, e.g., \cite[Lemma 2.2]{MS}). Consequently, we have the following inequalities 
	\begin{align*}
		&\dim\Hom_{\widetilde\g}(E, (U(\widetilde \g)/\Ann_{U(\widetilde \g)}(\widetilde M(s_\alpha\cdot \la)))^\ad) \\
		%&= \dim \Hom_\g(E, \mc L(M(\la),M(\la))\\
		&= \dim \Hom_{\widetilde\g}(E, \mc L(\widetilde M(\la),\widetilde M(\la))^\ad) \\
		&\geq \dim \Hom_{\widetilde\g}(E, \mc L(\widetilde M(s_\alpha\cdot\la),\widetilde M(s_\alpha\cdot\la))^\ad)\\
		&\geq \dim\Hom_{\widetilde\g}(E, (U(\widetilde \g)/\Ann_{U(\widetilde \g)}(\widetilde M(s_\alpha\cdot \la)))^\ad),
	\end{align*} for any projective-injective module $E$  in $\widetilde{\mc F}.$ This completes the proof for the first assertion. 
	
	Next, assume that $\langle \la+\rho,\alpha^\vee \rangle\not\in \Z$. Let $T_{s_\alpha}(-):\widetilde{\mc O}\rightarrow \widetilde{\mc O}$ denote the   
	Arkhipov's twisting functor from \cite[Section 3.6]{CMW} and \cite[Section 5]{CoM1}. Let $\widetilde{\mc O}_{\la+\Lambda}$ and $\widetilde{\mc O}_{s_{\alpha}\la+\Lambda}$ denote the full subcategory of $\widetilde{\mc O}$ consisting of all modules having weight supports contained in $\la+\Lambda$ and $s_\alpha\cdot\la+\Lambda$, respectively. In this case, $T_{s_\alpha}$ restricts to an equivalence between $\widetilde{\mc O}_{\la+\Lambda}$ and $\widetilde{\mc O}_{s_{\alpha} \la+\Lambda}$, sending $\widetilde M(\la)$ to $\widetilde M(s_\alpha\cdot\la)$. Since $T_{s_\alpha}$ commutes with tensoring with finite-dimensional weight $\widetilde\g$-modules, we have 
	\begin{align*}
		&\dim \Hom_{\widetilde \g}(E, \mc L(\widetilde M(\la),\widetilde M(\la))) = \dim \Hom_{\widetilde\g}(E, \mc L(\widetilde M(s_\alpha\cdot\la),\widetilde M(s_\alpha\cdot\la))),
	\end{align*} for any $E\in \widetilde{\mc F}$. Consequently, the answers to Kostant's problem for $\widetilde M(\la)$ and $\widetilde M(s_\alpha\cdot \la)$ are the same. This completes the proof.
\end{proof}

Now we are ready to prove the following result.
\begin{thm} \label{mainthm3}
	Let $\widetilde \g=\pn$ with  $\zeta\in \ch\mf n_+$. %Then the following are equivalent.  		\begin{itemize} \item[(1)] The Kostant's problem has an affirmative answer for $M(\la,\zeta)$.  			\item[(2)]  The Kostant's problem has an affirmative answer for $M_0(\la,\zeta)$ and $\la$ is typical.  		\end{itemize}
	Suppose that $\langle \la,\alpha^\vee\rangle \in \Z$, for any $\alpha\in \Pi_\zeta$. Then   $\widetilde M(\la,\zeta)$ is Kostant positive if and only if $\la$ is typical.
	%	In particular, if $\la$ is typical and $\langle \la+\rho,\alpha \rangle \in \Z$ for all $\alpha\in \Pi_\zeta$, then the Kostant's problem has an affirmative answer for $M(\la,\zeta)$.
\end{thm}
\begin{proof} By Theorem \ref{mainth3},  it suffices to prove the assertion in the case when $\zeta=0$. Assume that $\la$ is typical, then $\widetilde M(\la)$ is Kostant positive by \cite[Proof of Theorem 5.7]{Se02}. Conversely, assume that  $\widetilde M(\la)$ is Kostant positive.  By Proposition \ref{prop::16}, $\la$ is weakly typical.  The conclusion follows by Lemma \ref{lem::18}.  
\end{proof}

\subsection{Proof of Theorem C} 

\begin{cor} \label{cor::einThmC}
	Let $\widetilde \g$ be one of the Lie superalgebras $\gl(m|n)$, $\mf{osp}(2|2n)$ and $\pn$. Suppose that $\la\in \h^\ast$ is a typical and $W_\zeta$-anti-dominant wight such that $\langle \la, \alpha^\vee \rangle \in \Z$, for all $\alpha\in \Pi_\zeta$. Then  the answers to Kostant's problem for $\widetilde L(\la,\zeta)$, $\widetilde L(\la)$ and  $L(\la)$ are the same.
\end{cor}
\begin{proof}
	Suppose that $\widetilde \g =\gl(m|n)$ or $\mf{osp}(2|2n)$. From \cite[Corollary 6.8]{CM}, it follows that   $\widetilde L(\la) = K(L(\la))$ and $\widetilde L(\la,\zeta) = K(L(\la,\zeta))$, since $\Ann_{U(\g)}L(\la) = \Ann_{U(\g)} L(\la,\zeta)$ and  $\la$ is typical. The  conclusion follows by Theorem A and Theorem \ref{lem::16}. 	 Suppose that $\widetilde \g = \pn$.  We still have  $\widetilde L(\la) = K(L(\la))$,  since $\la$ is weakly typical. Applying the Backelin functor $\widetilde{\Gamma}_\zeta$ from subsection \ref{Sect::314}, we get that  $\widetilde L(\la,\zeta) = K(L(\la,\zeta))$. Again, $\Ann_{U(\g)}L(\la) = \Ann_{U(\g)} L(\la,\zeta)$. The conclusion follows by Theorems \ref{mainth3} and \ref{mainthm3}. This completes the proof.
\end{proof}

\begin{proof}[{\bf Proof of Theorem C from subsection \ref{sect::144}}]
The  conclusions in Part (1) follows by Theorems \ref{thm::13} and \ref{lem::16}. It remains to prove the assertions (c), (d) and (e) in  Part  (2). %We first prove the assertions in (c) and (d) in Part (2). 
In the case that  $\widetilde \g=\gl(m|n)$ or $\mf{osp}(2|2n)$, the conclusion (c) in Part (2) follows by Theorem A and Theorem \ref{thm::13}. In the case when $\widetilde \g=\pn$, the conclusion (c) in Part (2) is the same as the statement in Theorem \ref{mainthm3}. 
The conclusion (d) in Part (2) follows from Theorem \ref{mainth3} for all cases $\widetilde \g = \gl(m|n)$, $\mf{osp}(2|2n)$ or $\pn$. 

Finally, the conclusion (e) in Part (2)  follows from Corollary \ref{cor::einThmC}. This completes the proof of Theorem C.

\end{proof}

%\vspace{2mm}

%\noindent

\end{document}